\documentclass{amsart}
\usepackage{hyperref}
\newcommand*{\mailto}[1]{\href{mailto:#1}{\nolinkurl{#1}}}
\newcommand{\arxiv}[1]{\href{http://arxiv.org/abs/#1}{arXiv:#1}}

\newtheorem{theorem}{Theorem}[section]

\newtheorem{lemma}[theorem]{Lemma}
\newtheorem{corollary}[theorem]{Corollary}
\newtheorem{remark}[theorem]{Remark}
\newtheorem{hypothesis}[theorem]{Hypothesis}

\newcommand{\R}{{\mathbb R}}
\newcommand{\N}{{\mathbb N}}
\newcommand{\Z}{{\mathbb Z}}
\newcommand{\C}{{\mathbb C}}

\newcommand{\nn}{\nonumber}
\newcommand{\be}{\begin{equation}}
\newcommand{\ee}{\end{equation}}

\newcommand{\ol}{\overline}
\newcommand{\ti}{\tilde}

\newcommand{\spr}[2]{\langle #1 , #2 \rangle}

\newcommand{\E}{\mathrm{e}}
\newcommand{\I}{\mathrm{i}}

\newcommand{\im}{\mathrm{Im}}
\newcommand{\re}{\mathrm{Re}}
\newcommand{\dom}{\mathfrak{D}}

\DeclareMathOperator{\ran}{ran}

\newcommand{\floor}[1]{\lfloor#1 \rfloor}
\newcommand{\ceil}[1]{\lceil#1 \rceil}

\newcommand{\eps}{\varepsilon}

\newcommand{\sig}{\sigma}
\newcommand{\lam}{\lambda}
\newcommand{\gam}{\gamma}

\newcommand{\Lam}{\Lambda}
\newcommand{\ta}{\theta}

\numberwithin{equation}{section}

\begin{document}

\title[Weyl--Titchmarsh Theory for Strongly Singular Potentials]{Weyl--Titchmarsh Theory for Schr\"odinger Operators with Strongly Singular Potentials}

\author[A.\ Kostenko]{Aleksey Kostenko}
\address{Institute of Applied Mathematics and Mechanics\\
NAS of Ukraine\\ R. Luxemburg str. 74\\
Donetsk 83114\\ Ukraine\\ and School of Mathematical Sciences\\
Dublin Institute of Technology\\
Kevin Street\\ Dublin 8\\ Ireland}
\email{\mailto{duzer80@gmail.com}}

\author[A.\ Sakhnovich]{Alexander Sakhnovich}
\address{Faculty of Mathematics\\ University of Vienna\\
Nordbergstrasse 15\\ 1090 Wien\\ Austria}
\email{\mailto{Oleksandr.Sakhnovych@univie.ac.at}}
\urladdr{\url{http://www.mat.univie.ac.at/~sakhnov/}}

\author[G.\ Teschl]{Gerald Teschl}
\address{Faculty of Mathematics\\ University of Vienna\\
Nordbergstrasse 15\\ 1090 Wien\\ Austria\\ and International
Erwin Schr\"odinger
Institute for Mathematical Physics\\ Boltzmanngasse 9\\ 1090 Wien\\ Austria}
\email{\mailto{Gerald.Teschl@univie.ac.at}}
\urladdr{\url{http://www.mat.univie.ac.at/~gerald/}}

\thanks{Int. Math. Res. Not. {\bf 2012}, 1699--1747 (2012)}
\thanks{{\it Research supported by the Austrian Science Fund (FWF) under Grant No.\ Y330}}


\begin{abstract}
We develop Weyl--Titchmarsh theory for Schr\"odinger operators with strongly singular potentials such
as perturbed spherical Schr\"odinger operators (also known as Bessel operators). It is known that in such
situations one can still define a corresponding singular Weyl $m$-function and it was recently shown that
there is also an associated spectral transformation. Here we will give a general criterion when the
singular Weyl function can be analytically extended to the upper half plane. We will derive an integral
representation for this singular Weyl function and give a criterion when it is a generalized Nevanlinna
function. Moreover, we will show how essential supports for the Lebesgue decomposition of the
spectral measure can be obtained from the boundary behavior of the singular Weyl function.
Finally, we will prove a local Borg--Marchenko type uniqueness result.
Our criteria will in particular cover the aforementioned case
of perturbed spherical Schr\"odinger operators.
\end{abstract}

\maketitle

\section{Introduction}
\label{sec:int}

The present paper is concerned with spectral theory for one-dimensional Schr\"o\-dinger operators
\be
H = -\frac{d^2}{dx^2} + q(x), \qquad x\in (a,b),
\ee
on the Hilbert space $L^2(a,b)$ with a real-valued potential $q\in L^1_{\mathrm{loc}}(a,b)$.
If we assume one endpoint to be regular it is well known that one can associate a single
function $m(z)$, the Weyl--Titchmarsh (or Weyl) $m$-function, with $H$, such that $m(z)$ contains all the information about
$H$. In particular, a corresponding spectral measure $d\mu$ such that $H$ is unitarily equivalent to
multiplication by the identity function in $L^2(\R,d\mu)$ can be obtained by considering the vague limit of
$\pi^{-1} \im(m(\lam+\I\eps)) d\lam$ as $\eps\downarrow 0$. Moreover, by the celebrated Borg--Marchenko
theorem $H$ is uniquely determined by $m(z)$.

If neither endpoint is regular, similar results hold by considering
some point $c\in(a,b)$ at the price that one now has to deal with two by two Weyl matrices. This is clearly
unavoidable if the spectral multiplicity is two (for necessary and sufficient conditions for simplicity of the
spectrum we refer to \cite{gi}, see also the references therein), but it naturally raises the question if there are
cases where a single function is sufficient and a {\em singular} Weyl $m$-function $M(z)$ can be introduced.
One such case is if one endpoint is in the limit circle case which was first observed by Fulton \cite{ful77} and
later on also by Bennewitz and Everitt \cite{beev}.
Moreover, another example, the perturbed spherical Schr\"odinger (or Bessel) operator, got
much attention and singular $m$-functions were introduced by Gesztesy and Zinchenko \cite{gz}
 and Fulton and Langer \cite{ful08}, \cite{fl} recently (see also \cite{kl}). In this respect, note that the literature
on spectral theory for these operators is extensive, see, for example, \cite{ahm}, \cite{car}, \cite{gr}, \cite{kst}, \cite{se} and
the references therein.

The key ingredient for defining a Weyl $m$-function is an entire system of linearly independent
solutions $\phi(z,x)$ and $\theta(z,x)$ of the underlying differential equation $H u = z u$, $z\in\C$,
normalized such that the Wronskian $W(\theta(z),\phi(z))$ equals $1$. To make the connection with $H$,
one solution, say $\phi(z,x)$, has to be chosen such that it lies in the domain of $H$ near the endpoint $a$
(i.e., $\phi(z,.)$ is square integrable near $a$ and satisfies the boundary condition at $a$ if $H$ is limit circle at $a$).
If $a$ is regular, this can be easily done by specifying the initial conditions at $x=a$. Once $\phi(z,x)$ and
$\theta(z,x)$ are given, the Weyl $m$-function $M(z)$ can be defined by the requirement that the solution
\be
\psi(z,x) = \theta(z,x) + M(z) \phi(z,x)
\ee
is in the domain of $H$ near $b$.

While this prescription sounds straightforward, it has turned out to be rather subtle! The definition of
$\phi(z,x)$ is unproblematic and, as demonstrated in Gesztesy and Zinchenko \cite{gz}, a necessary
and sufficient condition for the existence of $\phi(z,x)$ is the assumption that $H$ restricted to $(a,c) \subset (a,b)$
has purely discrete spectrum. This is also confirmed by the fact that in the special case of spherical Schr\"odinger
operators the solution $\phi(z,x)$ (together with detailed growth estimates) can be obtained by the usual iteration schemes
(see, e.g.\ \cite[Sect.~14]{fad}, \cite{gr}, \cite{kst}, \cite[Sect.~12.2]{new}, \cite{se}).
On the other hand, a similar construction for the second (singular) solution $\theta(z,x)$ is not known
to the best of our knowledge! More precisely, a construction was given in \cite{gr} and later on used and generalized by
several authors, but unfortunately this construction is wrong (cf.\ \cite[Rem.~2.7]{kst} for a counterexample). Consequently,
this creates a gap in the proofs of all papers relying on this construction (for \cite{gr} and \cite{se} this has been fixed in \cite{gr2}).
The only case of spherical Schr\"odinger operators, where existence of $\theta(z,x)$ is known, is when the potential is analytic
of Fuchs type near $a=0$. In this case a second entire solution comes for free from the Frobenius method. Naturally, this special
case has attracted much attention recently \cite{ful08}, \cite{fl}, \cite{kl}.

This clearly raises the question when such a second solution $\theta(z,x)$ exists or (essentially) equivalent,
when the singular Weyl $m$-function is analytic in the entire upper (and hence lower) half plane. This question was
circumvented in \cite{gz} by observing that for the definition of an associated spectral measure,
all one needs is a definition of $\theta(z,x)$ for $z$ in a vicinity of the real line. Such a $\theta(z,x)$ is easy to
obtain and existence of an associated spectral transformation was established in \cite{gz}. However, analyticity
of the Weyl $m$-function in the upper half plane is not only interesting from an academic point of view,
but also from a practical perspective! In fact, the Borg--Marchenko uniqueness theorem involves the asymptotic
expansion of the Weyl $m$-function as $|z| \to \infty$ along nonreal rays. Hence analyticity in the upper
half plane is a vital prerequisite for such kind of results.

It is the purpose of the present paper to give a positive answer to the above raised questions:
As our first result we will show that a second (singular) entire solution $\theta(z,x)$ always exists in Section~\ref{sec:swm}.
This will be done using an explicit construction based on the Mittag--Leffler theorem. Next we will
reproduce the main result concerning existence of a spectral transformation from \cite{gz} in Section~\ref{sec:st}.
In fact, we will provide some generalizations and a simplified
proof. Moreover, we will show how essential supports for the Lebesgue decomposition of the
spectral measure can be obtained from the boundary behavior of the singular Weyl function (cf.\ Corollary~\ref{corsup}).

In Section~\ref{sec:ir} we will derive an integral representation for singular Weyl functions which generalizes the usual
Herglotz--Nevanlinna representation (cf.\ Theorem~\ref{IntR}). Moreover, with its help we can show that there is always a choice of
$\phi(z,x)$ and $\theta(z,x)$ such that the singular Weyl function is a Herglotz--Nevanlinna function. In fact, in the special case
of perturbed spherical Schr\"odinger operators quite some effort has been put into showing that the singular Weyl function is a
generalized Nevanlinna function \cite{DSh_00}, \cite{fl}, \cite{kl}, but so far only the case of analytic perturbations
could be treated. Our approach is both more general and much simpler. In fact, as a simple special case, it implies that when
$\phi(z,x)$ is normalized such that it asymptotically matches the Bessel function of the first kind at the singular endpoint, then
there is a corresponding choice of $\theta(z,x)$ such that the singular Weyl function is in the generalized Nevanlinna class $N_\kappa^\infty$
(cf.\ Theorem~\ref{thm:nkap}, Theorem~\ref{thmpbgn} and Appendix~\ref{app:gnf} for a definition of this class).

In Section~\ref{sec:ex} the explicitly solvable example of the (unperturbed) Bessel operator will be used to
illustrate some of our findings.

Then we will investigate exponential growth properties of solutions as a preparation
for our next section. Existence of solutions with the necessary growth properties will require
an additional (mild) hypothesis on the Dirichlet and Neumann eigenvalues of the operator
restricted to $(a,c) \subset (a,b)$. Our approach will be nonconstructive and based on a version of the
Corona Theorem for rings of entire functions due to H\"ormander \cite{hor}.
After these preparations we will prove a Borg--Marchenko-type uniqueness result in its local version
due to Gesztesy and Simon \cite{gs}, \cite{si} (our proof is modeled after the simple proof from Bennewitz \cite{ben}).

Finally we will show in Section~\ref{sec:ap} how our results can be applied to perturbed spherical Schr\"odinger operators.

Appendix~\ref{app:lc} shows how our considerations simplify in the special case where the endpoint is limit circle and
shows that the singular $m$-function is a Herglotz--Nevanlinna function in this case. Appendix~\ref{app:ex} investigates another
explicitly solvable example. Appendix~\ref{app:gnf} collects some facts about generalized Nevanlinna functions.

In a follow-up paper \cite{kt} we will further investigate the special example of perturbed spherical Schr\"odinger operators
\be
H = -\frac{d^2}{dx^2} + \frac{l(l+1)}{x^2} + q(x), \qquad x\in (0,\infty).
\ee
There we will precisely identify the value of $\kappa$ from the generalized Nevanlinna class $N_\kappa^\infty$ to which
the singular $m$-function belongs and connect our results with the theory of super singular perturbations \cite{be}, \cite{DHdS},
\cite{DKSh_05}, \cite{DLSZ}, \cite{Ku}, \cite{sho}. Let us mention that this connection was first observed by Djiksma and Shondin \cite{DSh_00}
and Kurasov and Luger \cite{kl} in the particular cases of unperturbed Bessel operators and the Bessel operators perturbed by a Coulomb term, respectively.
In the case where the perturbation is analytic similar results have been obtained in \cite{fl} with the help of the Frobenius method as pointed out earlier.

There are also close connections with commutation methods (Crum--Darboux transformations) originating in the work of Krein \cite{kr57}.
This connection will be explored in a follow-up paper \cite{kst2}.

\section{The singular Weyl $m$-function}
\label{sec:swm}

To set the stage, we will consider one-dimensional Schr\"odinger operators on $L^2(a,b)$
with $-\infty \le a<b \le \infty$ of the form
\begin{equation} \label{stli}
\tau = - \frac{d^2}{dx^2} + q,
\end{equation}
where the potential $q$ is real-valued satisfying
\begin{equation}
q \in L^1_{loc}(a,b).
\end{equation}
We will use $\tau$ to denote the formal differential expression and $H$ to denote a corresponding
self-adjoint operator given by $\tau$ with separated boundary conditions at $a$ and/or $b$.

If $a$ (resp.\ $b$) is finite and $q$ is in addition integrable
near $a$ (resp.\ $b$), we will say $a$ (resp.\ $b$) is a \textit{regular}
endpoint.  We will say $\tau$, respectively $H$, is \textit{regular} if
both $a$ and $b$ are regular.

We will choose a point $c\in(a,b)$ and also consider the operators $H^D_{(a,c)}$ and $H^D_{(c,b)}$
which are obtained by restricting $H$ to $(a,c)$ and $(c,b)$ with a Dirichlet boundary condition at
$c$, respectively. The corresponding operators with a Neumann boundary condition will be
denoted by $H^N_{(a,c)}$ and $H^N_{(c,b)}$.

Moreover, let $c(z,x)$ and $s(z,x)$ be the solutions of $\tau u = z\, u$ corresponding
to the initial conditions $c(z,c)=1$, $c'(z,c)=0$ and $s(z,c)=0$, $s'(z,c)=1$.

Then we can define the Weyl solutions
\begin{align}\nn
u_-(z,x) &= c(z,x) - m_-(z) s(z,x), \qquad z\in\C\setminus\sig(H^D_{(a,c)}),\\\label{defupm}
u_+(z,x) &= c(z,x) + m_+(z) s(z,x), \qquad z\in\C\setminus\sig(H^D_{(c,b)}),
\end{align}
where $m_\pm(z)$ are the Weyl $m$-functions corresponding to the base point $c$ and associated with
$H^D_{(a,c)}$, $H^D_{(c,b)}$, respectively.
For further background we refer to \cite[Chap.~9]{tschroe} or \cite{wdln}.

To define an analogous singular Weyl $m$-function at the, in general singular, endpoint $a$ we will
first need the analog of the system of solutions $c(z,x)$ and $s(z,x)$. Hence our first goal is
to find a system of entire solutions $\theta(z,x)$ and $\phi(z,x)$ such that $\phi(z,x)$ lies in the domain of
$H$ near $a$ and such that the Wronskian $W(\theta(z),\phi(z))=1$. To this end we start with a
hypothesis which will turn out necessary and sufficient for such a system of solutions to exist.

\begin{hypothesis}\label{hyp:gen}
Suppose that the spectrum of $H^D_{(a,c)}$ is purely discrete for one (and hence for all)
$c\in (a,b)$.
\end{hypothesis}

Note that this hypothesis is for example satisfied if $q(x) \to +\infty$ as $x\to a$ (cf.\ Problem~9.7 in \cite{tschroe}).
Further examples will be given in Sections~\ref{sec:ex} and \ref{sec:ap} (cf.\ also Example~3.13 in \cite{gz}).

\begin{lemma}\label{lem:phi}
Hypothesis~\ref{hyp:gen} is necessary and sufficient to have a (nontrivial) solution $\phi(z,x)$ of $\tau u = z u$
which is in the domain of $H$ near $a$ and which is entire with respect to $z$.
Moreover, one can choose $\phi(z,x)$ such that
\be\label{phiab}
\phi(z,x) = \alpha(z) c(z,x) + \beta(z) s(z,x),
\ee
where $\alpha(z)$ and $\beta(z)$ are real entire functions with no common zeros.
\end{lemma}

\begin{proof}
Suppose Hypothesis~\ref{hyp:gen} holds. Then $m_-(z)$ is meromorphic with simple poles at the points in $\sig(H^D_{(a,c)})$
and by the Weierstra{\ss} product theorem (\cite[Thm.~II.10.1]{mar}) there is a (real) entire function $\alpha(z)$ which has simple
zeros (and no other zeros) at the poles of $m_-(z)$. Hence we can set $\phi(z,x)=\alpha(z) u_-(z,x)$.

Conversely, let $\phi(z,x) = \alpha(z) c(z,x) + \beta(z) s(z,x)$ be entire. Then $\alpha(z) = \phi(z,c)$ is also entire.
Moreover,
\[
\beta(z) = \frac{\phi(z,x) - \alpha(z) c(z,x)}{s(z,x)}
\]
is entire since the possible poles on the real line are removable (use that the left-hand side of this formula is
independent of $x$ and the possible poles on the right-hand side vary as $x$ varies).
Note that
\be\label{eqalbe}
\alpha(z) = \phi(z,c), \qquad \beta(z) = \phi'(z,c)
\ee
Finally, recalling \eqref{defupm}, we see that
\be
m_-(z) = -\frac{\phi'(z,c)}{\phi(z,c)}=-\frac{\beta(z)}{\alpha(z)}
\ee
is meromorphic and thus Hypothesis~\ref{hyp:gen} holds.
\end{proof}

\begin{corollary}\label{coruniqphi}
The function $\phi(z,x)$ is uniquely defined up to a real entire function without zeros. That is, if
$\ti{\phi}(z,x)$ is another real entire solution which is nontrivial for all $z\in\C$ and in the domain of $H$ near $a$,
we have
\be
\ti{\alpha}(z) = \E^{g(z)} \alpha(z), \qquad \ti{\beta}(z) = \E^{g(z)} \beta(z)
\ee
for some real entire function $g(z)$.
\end{corollary}

It remains to find a second solution:

\begin{lemma}\label{lem:theta}
Suppose Hypothesis~\ref{hyp:gen} holds. Then there is a second solution
\be\label{thetagd}
\theta(z,x) = \gamma(z) c(z,x) + \delta(z) s(z,x),
\ee
where $\gamma(z)$ and $\delta(z)$ are real entire functions with no common zeros,
such that
\be
W(\theta(z),\phi(z)) = \gamma(z) \beta(z) - \alpha(z) \delta(z) =1.
\ee
Here $W(u,v)= u(x)v'(x)-u'(x)v(x)$ is the usual Wronski determinant.
\end{lemma}

\begin{proof}
We will make the following ansatz
\[
\gamma(z) = \frac{\beta(z)}{\alpha(z)^2+\beta(z)^2} - \eta(z) \alpha(z), \qquad
\delta(z)= \frac{-\alpha(z)}{\alpha(z)^2+\beta(z)^2} - \eta(z) \beta(z),
\]
where $\eta(z)$ is a meromorphic function with poles at the zeros of $\alpha(z)^2+\beta(z)^2$ to
be determined next.

Denote by $\{ z_j\}$ the zeros of the entire function $\alpha(z)^2+\beta(z)^2$ and suppose that
$z_j$ is a zero of order $n_j\in\N$. Then at each $z_j$ we have
\[
\beta^{(k)}(z_j) = \sig_j \I \alpha^{(k)}(z_j), \qquad 0 \le k < n_j,
\]
with some $\sig_j\in \{\pm 1\}$. Now choose $\eta(z)$ such that the principal part of $\eta(z)$ near
$z_j$ matches the one from $\sig_j \I (\alpha(z)^2+\beta(z)^2)^{-1}$, that is,
\[
\frac{\sig_j \I}{\alpha(z)^2+\beta(z)^2} = \eta(z) + O(z-z_j)^0.
\]
Such a function exists by the Mittag--Leffler theorem \cite[Thm.~II.10.10]{mar}. Then one computes
\begin{align*}
\gamma(z) &= \frac{\beta(z)}{\alpha(z)^2+\beta(z)^2} - \eta(z) \alpha(z)\\
 &= \frac{\sig_j \I \alpha(z) + O(z-z_j)^{n_j}}{\alpha(z)^2+\beta(z)^2} -  \frac{\sig_j \I \alpha(z)}{\alpha(z)^2+\beta(z)^2} + O(z-z_j)^0
= O(z-z_j)^0
\end{align*}
and thus all poles of $\gamma(z)$ are removable. Similarly, all poles of $\delta(z)$ are removable
\begin{align*}
\delta(z) &= \frac{-\alpha(z)}{\alpha(z)^2+\beta(z)^2} - \eta(z) \beta(z)\\
 &= \frac{-\alpha(z)}{\alpha(z)^2+\beta(z)^2} -  \frac{(\sig_j \I)^2 \alpha(z) + O(z-z_j)^{n_j}}{\alpha(z)^2+\beta(z)^2} + O(z-z_j)^0
= O(z-z_j)^0.
\end{align*}
Finally, $W(\theta(z),\phi(z)) =1$ is immediate (at least away from the zeros) by our ansatz. If $\gamma(z)$ or $\delta(z)$ are
not real, then replace them by $\frac{1}{2}(\gamma(z) + \gamma(z^*)^*)$ and $\frac{1}{2}(\delta(z) + \delta(z^*)^*)$, respectively.
\end{proof}

Note again
\be
\gamma(z) = \theta(z,c), \qquad \delta(z) = \theta'(z,c).
\ee

\begin{corollary}\label{coruniqtheta}
Given $\phi(z,x)$ and $\theta(z,x)$, any other real entire solution $\ti{\theta}(z,x)$ satisfying $W(\ti{\theta}(z),\phi(z))=1$
is given by
\be
\ti{\theta}(z,x) = \theta(z,x) - f(z) \phi(z,x)
\ee
for some real entire function $f(z)$.
\end{corollary}

Given a system of real entire solutions $\phi(z,x)$ and $\theta(z,x)$ as in the above lemma we can define the
singular Weyl $m$-function
\be\label{defM}
M(z) = -\frac{W(\theta(z),u_+(z))}{W(\phi(z),u_+(z))} =
-\frac{\gamma(z) m_+(z) - \delta(z)}{\alpha(z) m_+(z) - \beta(z)}
\ee
such that the solution which is in the domain of $H$ near $b$ (cf.\ \eqref{defupm}) is given by
\be
u_+(z,x)= a(z) \big(\theta(z,x) + M(z) \phi(z,x)\big),
\ee
where $a(z)= - W(\phi(z),u_+(z)) = \beta(z) - m_+(z) \alpha(z)$.
By construction we obtain the following:

\begin{lemma}
The singular Weyl $m$-function $M(z)$ is analytic in $\C\backslash\R$ and satisfies $M(z)=M(z^*)^*$.
\end{lemma}

\begin{proof}
It is clear from \eqref{defM} that $M(z)$ is meromorphic in $\C\backslash\R$ since $m_+(z)$ is
analytic in $\C\backslash\R$ and all other functions are entire. Moreover, $W(\phi(z),u_+(z))$
cannot vanish for $z\in\C\backslash\R$ since every such zero would correspond to a complex
eigenvalue of $H$. Hence the claim follows.
\end{proof}

Rather than $u_+(z,x)$ we will use
\be\label{defpsi}
\psi(z,x)= \theta(z,x) + M(z) \phi(z,x).
\ee

\begin{remark}\label{rem:uniq}
By Corollary~\ref{coruniqphi} and Corollary~\ref{coruniqtheta} any system of real entire solutions
$\ti{\theta}(z,x)$ and $\ti{\phi}(z,x)$ satisfying $W(\ti{\theta}(z),\ti{\phi}(z))=1$ is related to those
constructed above via
\[
\ti{\theta}(z,x) = \E^{-g(z)} \theta(z,x) - f(z) \phi(z,x), \qquad
\ti{\phi}(z,x) = \E^{g(z)} \phi(z,x),
\]
where $g(z)$ and $f(z)$ are real entire functions. The
singular Weyl $m$-functions are related via
\[
\ti{M}(z) = \E^{-2g(z)} M(z) + \E^{-g(z)}f(z).
\]
In particular, the maximal domain of holomorphy or the structure of poles and singularities
do not change.
\end{remark}

\section{Spectral transformations}
\label{sec:st}

Denote by $f\in L^2_c(a,b)$ the subset of square integrable functions with compact support.
Set
\be\label{defhatf}
\hat{f}(z) = \int_a^b \phi(z,x) f(x) dx, \qquad f \in L^2_c(a,b)
\ee
and note that $\hat{f}(z)$ is entire. Note
\be
\hat{f}(z^*)^* = \int_a^b \phi(z,x) f(x)^* dx.
\ee
Moreover, recall that for every $f\in L^2(a,b)$ there is an associated spectral measure
$\mu_f$ whose Borel transform is given by
\be
m_f(z) := \spr{f}{(H-z)^{-1} f} = \int_\R \frac{d\mu_f(\lam)}{\lam-z}.
\ee
Here $\spr{f}{g} = \int_a^b f(x)^* g(x) dx$ denotes the scalar product in $L^2(a,b)$.

Moreover, recall
\be
(H-z)^{-1} f(x) = \int_a^b G(z,x,y) f(y) dy,
\ee
where
\be\label{defgf}
G(z,x,y) = \begin{cases} \phi(z,x) \psi(z,y), & y\ge x,\\
\phi(z,y) \psi(z,x), & y\le x,\end{cases}
\ee
is the Green function of $H$ (cf.\ \cite[Lem.~9.7]{tschroe}).

\begin{lemma}\label{lemmfM}
For every $f \in L^2_c(a,b)$ we have
\be
m_f(z) = E_f(z) + \hat{f}(z) \hat{f}(z^*)^* M(z),
\ee
where $\hat{f}(z)$ is given by \eqref{defhatf} and $E_f(z)$ is entire and satisfies $E_f(z^*)^* = E_f(z)$.
\end{lemma}

\begin{proof}
A straightforward calculation using
\begin{align*}
&\spr{f}{(H-z)^{-1} f} = \int_a^b f(x)^* \int_a^b G(z,x,y) f(y) dy\, dx\\
& \quad = \int_a^b f(x)^* \left( \psi(z,x) \int_a^x \phi(z,y) f(y) dy + \phi(z,x) \int_x^b \psi(z,y) f(y) dy\right) dx
\end{align*}
and \eqref{defpsi} proves the claim with
\[
E_f(z) = \int_a^b f(x)^* \left( \theta(z,x) \int_a^x \phi(z,y) f(y) dy + \phi(z,x) \int_x^b \theta(z,y) f(y) dy\right) dx.
\]
\end{proof}

For the proof of our main result we will need the following form of the Stieltjes inversion formula:

\begin{lemma}[Stieltjes inversion formula]\label{lem:sif}
Suppose $m(z)$ is the Borel transform of a finite measure $d\mu$,
\be
m(z) = \int_\R \frac{d\mu(\lam)}{\lam-z},
\ee
then
\be
\lim_{\eps\downarrow 0} \frac{1}{\pi} \int_{\lam_0}^{\lam_1} F(\lam) \im\big(m(\lam+\I\eps)\big) d\lam =
\int_{\lam_0}^{\lam_1} F(\lam) d\mu(\lam)
\ee
for every $F\in C[\lam_0,\lam_1]$, where
\be
\int_{\lam_0}^{\lam_1} F\, d\mu = \frac{1}{2} \left(\int_{(\lam_0,\lam_1)} F\, d\mu +  \int_{[\lam_0,\lam_1]} F\, d\mu\right).
\ee
\end{lemma}

We also recall that we have
\be\label{bbhf}
\eps |m(\lam+\I\eps)| \le \mu(\R) \qquad\text{and} \qquad \lim_{\eps\downarrow 0} (-\I\eps)\, m(\lam+\I\eps) =\mu(\{\lam\}).
\ee

Upon choosing $f(x)= \chi_{[c,d]}(x) \phi(\lam_0,x)$ for some fixed $\lam_0\in\R$
we have $\hat{f}(\lam_0)>0$ and obtain
\be\label{eqMEm}
M(z) = \frac{-E_f(z) + \spr{f}{(H-z)^{-1} f}}{\hat{f}(z)^2}
\ee
for $z$ in a vicinity of $\lam_0$. Thus $M(z)$ shares many properties of the Herglotz--Nevanlinna function
$\spr{f}{(H-z)^{-1} f}$. In particular, we can use this to associate a measure with $M(z)$ by virtue of
the Stieltjes--Liv\v{s}i\'{c} inversion formula (\cite[Lem.~2.1]{kk}).

\begin{lemma}\label{lemrho}
There is a unique Borel measure $d\rho$ defined via
\be\label{defrho}
\frac{1}{2} \left( \rho\big((\lam_0,\lam_1)\big) + \rho\big([\lam_0,\lam_1]\big) \right)=
\lim_{\eps\downarrow 0} \frac{1}{\pi} \int_{\lam_0}^{\lam_1} \im\big(M(\lam+\I\eps)\big) d\lam
\ee
such that
\be\label{rhomuf}
d\mu_f = |\hat{f}|^2 d\rho, \qquad f\in L^2_c(a,b),
\ee
where $d\mu_f$ is the spectral measure of $f$ associated with the self-adjoint operator $H$.
\end{lemma}

\begin{proof}
Fix $\lam_0<\lam_1$ and $f\in L^2_c(a,b)$ such that $\hat{f}(\lam) \ne 0$ for $\lam\in[\lam_0,\lam_1]$. Then
\eqref{eqMEm} implies
\begin{align*}
\lim_{\eps\downarrow 0} \frac{1}{\pi} \int_{\lam_0}^{\lam_1} F(\lam) \im\big(M(\lam+\I\eps)\big) d\lam &=
\lim_{\eps\downarrow 0} \frac{1}{\pi} \int_{\lam_0}^{\lam_1} F(\lam) \im\left(\frac{m_f(\lam+\I\eps)}{\hat{f}(\lam+\I\eps) \hat{f}(\lam-\I\eps)^*}\right) d\lam\\
&= \lim_{\eps\downarrow 0} \frac{1}{\pi} \int_{\lam_0}^{\lam_1} \frac{F(\lam)}{|\hat{f}(\lam)|^2}\im\big(m_f(\lam+\I\eps)\big) d\lam\\
&= \int_{\lam_0}^{\lam_1} \frac{F(\lam)}{|\hat{f}(\lam)|^2} d\mu_f(\lam).
\end{align*}
Here the first step follows from dominated convergence using \eqref{bbhf} and the second from Lemma~\ref{lem:sif}.
Hence we can choose $F(\lam)=1$ to obtain \eqref{defrho} (split the interval into smaller subintervals
and use different $f$'s for different subintervals if necessary).

Moreover, replacing $F(\lam)$ by $|\hat{f}(\lam)|^2F(\lam)$ we obtain
\[
\int_\R F |\hat{f}|^2 d\rho = \int_\R F d\mu_f
\]
for every continuous function with compact support away from the real zeros of $\hat{f}(\lam)$ (note that the
zeros are discrete since $\hat{f}(z)$ is entire). Next observe that at every real zero $\lam_0$ of $\hat{f}$ we
have $\mu_f(\{\lam_0\})=0$ (if $\mu_f(\{\lam_0\})>0$, then $\lam_0$ must be an eigenvalue of $H$ with
corresponding eigenfunction $\phi(\lam_0,.)$ and $\mu_f(\{\lam_0\}) = |\spr{\phi(\lam_0)}{f}|^2 / \|\phi(\lam_0)\|^2
=|\hat{f}(\lam_0)|^2/ \|\phi(\lam_0)\|^2$ contradicting $\hat{f}(\lam_0)=0$) and hence we can remove this
restriction implying \eqref{rhomuf}.
\end{proof}

Now we come to our main result in this section which summarizes the main result from \cite{gz}.

\begin{theorem}[\cite{gz}]
Define
\be
\hat{f}(\lam) = \lim_{c\uparrow b} \int_a^c \phi(\lam,x) f(x) dx,
\ee
where the right-hand side is to be understood as a limit in $L^2(\R,d\rho)$. Then the map
\be
U: L^2(a,b) \to L^2(\R,d\rho), \qquad f \mapsto \hat{f},
\ee
is unitary and its inverse is given by
\be\label{Uinv}
f(x) = \lim_{r\to\infty} \int_{-r}^r \phi(\lam,x) \hat{f}(\lam) d\rho(\lam),
\ee
where again the right-hand side is to be understood as a limit in $L^2(a,b)$.
Moreover, $U$ maps $H$ to multiplication by $\lam$.
\end{theorem}

\begin{proof}

Equation \eqref{rhomuf} implies
\[
\|f\|^2 = \int_\R d\mu_f = \int_\R |\hat{f}|^2 d\rho = \|\hat{f}\|^2_\rho, \quad f\in L^2_c(a,b),
\]
and thus the unique extension of this map to $L^2(a,b)$ is isometric. Here we have used
$\|.\|_\rho$, $\spr{.}{.}_\rho$ to distinguish the norm, scalar product in $L^2(\R,d\rho)$ from the one in $L^2(a,b)$.
Moreover,
\[
\spr{f}{F(H) f} = \int_\R F d\mu_f = \int_\R F |\hat{f}|^2 d\rho = \spr{\hat{f}}{F \hat{f}}_\rho, \qquad
f\in L^2_c(a,b),
\]
for every bounded Borel function $F$. By polarization and approximation we even get
\[
\spr{f}{F(H) g}  = \spr{\hat{f}}{F \hat{g}}_\rho, \qquad f,g\in L^2(a,b),
\]
for every bounded Borel function $F$.

Now consider $f,g\in L^2(a,b)$ and two bounded Borel functions
$F,G$. Set $h= F(H) g$. Then the above equation (with $F \to G F$) implies
\[
\int_\R G \hat{f}^* (\hat{h} - F \hat{g}) d\rho =0, \qquad h = F(H) g,
\]
for every bounded Borel function $G$ and thus $\hat{f}(\lam)^* (\hat{h}(\lam) - F \hat{g}(\lam))=0$ for
$\rho$-a.e.\ $\lam$. Furthermore, since for every $\lam_0$ we can find an $f$ such that $\hat{f}(\lam_0)\ne 0$,
we even get $\hat{h} = F \hat{g}$. This shows that $\ran(U)$ contains, for example, all characteristic functions of
intervals and thus $\ran(U)= L^2(\R,d\rho)$.
\end{proof}

Moreover, the spectral types can be read off from the boundary behavior of the
singular Weyl function in the usual way.

\begin{corollary}\label{corsup}
The following sets
\begin{align} \nn
\Sigma_{ac} &= \{\lam | 0<\limsup_{\eps\downarrow 0} \im(M(\lam+\I\eps)) < \infty\},\\
\Sigma_s &= \{\lam | \limsup_{\eps\downarrow 0}\im(M(\lam+\I\eps)) = \infty\},\\ \nn
\Sigma_p &= \{\lam | \lim_{\eps\downarrow 0} \eps\im(M(\lam+\I\eps))>0 \},\\
\Sigma &= \Sigma_{ac} \cup \Sigma_s = \{\lam | 0<\limsup_{\eps\downarrow 0} \im(M(\lam+\I\eps))\}
\end{align}
are minimal supports for $\rho_{ac}$, $\rho_s$, $\rho_{pp}$, and $\rho$, respectively.
In fact, we could even restrict ourselves to values of $\lam$, where the $\limsup$ is a $\lim$ (finite or infinite).

Moreover, the spectrum of $H$ is given by the closure of $\Sigma$,
\be
\sig(H) = \ol{\Sigma},
\ee
the set of eigenvalues is given by
\be
\sig_p(H) = \Sigma_p,
\ee
and the absolutely continuous spectrum of $H$ is given by the essential closure
of $\Sigma_{ac}$,
\be
\sig(H_{ac}) = \ol{\Sigma}_{ac}^{ess}.
\ee
Recall $ \ol{\Omega}^{ess} = \{ \lam\in\R | |(\lam-\eps,\lam+\eps)\cap \Omega|>0
\mbox{ for all } \eps>0\}$, where $|\Omega|$ denotes the Lebesgue measure of a Borel set $\Omega$.
\end{corollary}

\begin{proof}
It suffices to show this result restricted to sufficiently small intervals $[\lam_0,\lam_1]$. Now choose
$f\in L^2_c(a,b)$ such that $\hat{f}(\lam) \ne 0$ for $\lam\in[\lam_0,\lam_1]$ as in the previous lemma.

Then, by Lemma~\ref{lemmfM}, the above sets (restricted to $[\lam_0,\lam_1]$) remain unchanged if we replace $M(z)$ by the Herglotz--Nevanlinna function
$m_f(z)$. Moreover, the measures $\mu_f$ and $\rho$ are mutually absolutely continuous on $[\lam_0,\lam_1]$ and
thus the claim follows from standard results (cf., e.g., \cite[Sect.~3.2]{tschroe}).
\end{proof}

In particular, note that
\be\label{defnc}
\lim_{z\to\lam} (\lam-z) M(z) = \left( \int_a^b \phi(\lam,x)^2 dx \right)^{-1} \ge 0
\ee
and minus the residue at an eigenvalue is given by the corresponding norming constant
as usual.

We conclude this section with one more simple but useful observation.

\begin{lemma}\label{lemUub}
Recall the Green function $G(z,x,y)$ of $H$ defined in \eqref{defgf}. Then
\be\label{UG}
(U G(z,x,.))(\lam) = \frac{\phi(\lam,x)}{\lam-z} \quad\text{and}\quad
(U \partial_x G(z,x,.))(\lam) = \frac{\phi'(\lam,x)}{\lam-z}
\ee
for every $x\in(a,b)$ and every $z\in\C\setminus\sig(H)$.
\end{lemma}

\begin{proof}
First of all note that $G(z,x,.)\in L^2(a,b)$ for every $x\in(a,b)$ and $z\in\C\setminus\sig(H)$.
Moreover, from $(H-z)^{-1} f = U^{-1} \frac{1}{\lam-z} U f$ we have
\be \label{eqslgfsptr}
\int_a^b G(z,x,y) f(y)\, dy = \int_\R \frac{\phi(\lam,x) \hat{f}(\lam)}{\lam-z} d\rho(\lam).
\ee
Here equality is to be understood in $L^2$, that is, for a.e.\ $x$ (cf.\ \eqref{Uinv}).
However, the left-hand side is continuous with respect to $x$ and so is the
right-hand side, at least if $\hat{f}$ has compact support. Since this set is dense,
the first equality in \eqref{UG} follows. Similarly, the second follows after differentiating \eqref{eqslgfsptr}
with respect to $x$.
\end{proof}

Differentiating with respect to $z$ we even obtain

\begin{corollary}
We have
\be
(U \partial_z^k G(z,x,.))(\lam) = \frac{k! \phi(\lam,x)}{(\lam-z)^{k+1}} \:\text{and}\:
(U \partial_z^k \partial_x G(z,x,.))(\lam) = \frac{k! \phi'(\lam,x)}{(\lam-z)^{k+1}}
\ee
for every $x\in(a,b)$, $k\in\N_0$, and every $z\in\C\setminus\sig(H)$.
\end{corollary}

\begin{remark}\label{rem:uniqrho}
We have seen in Remark~\ref{rem:uniq} that $M(z)$ is not unique. However, given
$\ti{M}(z)$ as in Remark~\ref{rem:uniq}, the spectral measures are related by
\[
d\ti{\rho}(\lam) = \E^{-2g(\lam)} d\rho(\lam).
\]
Hence the measures are mutually absolutely continuous and the associated spectral
transformation just differ by a simple rescaling with the positive function $\E^{-2g(\lam)}$.
\end{remark}

\section{An integral representation for singular $m$-functions}
\label{sec:ir}

One of the most important properties of Herglotz--Nevanlinna functions is the existence of
an integral representation. Our first result gives such an integral representation for our singular $m$-function.

\begin{theorem}\label{IntR}
Let $M(z)$ be a singular Weyl function and $\rho$ its associated spectral measure. Then there exists
an entire function $g(z)$ such that $g(\lam)\ge 0$ for $\lam\in\R$ and $\E^{-g(\lam)}\in L^2(\R, d\rho)$.

Moreover, for any entire function $\hat{g}(z)$ such that $\hat{g}(\lam)>0$ for $\lam\in\R$ and $(1+\lam^2)^{-1} \hat{g}(\lam)^{-1}\in L^1(\R, d\rho)$
(e.g.\ $\hat{g}(z)=\E^{2g(z)}$) we have the integral representation
\be\label{Mir}
M(z) = E(z) + \hat{g}(z) \int_\R \left(\frac{1}{\lam-z} - \frac{\lam}{1+\lam^2}\right) \frac{d\rho(\lam)}{\hat{g}(\lam)},
\qquad z\in\C\backslash\sig(H),
\ee
where $E(z)$ is a real entire function.
\end{theorem}

\begin{proof}
We first show that entire functions $\hat{g}$ of the required type exist. To this end we
define the function
\[
R(\lam) = \begin{cases} \int_{[-\lam,\lam]} d\rho, & \lam \ge 0,\\
0, & \lam <0, \end{cases}
\]
which is nonnegative and nondecreasing for $\lam>0$. Then we can find an entire function $h(z) = \sum_{j=0}^\infty h_j z^j$ such that
$h(n^2) = R(n+1)$ for $n\in\N_0$ (cf.\ Example~1 in Sect.~51 of \cite{mar}). Now choose
\[
g(z) =  \frac{1}{2}\sum_{j=0}^\infty |h_j| z^{2j}.
\]
Then, by construction $R(n+1) = h(n^2) \le 2g(n)$ and thus $R(\lam) \le 2g(\lam)$ for $\lam\ge 0$. Moreover,
\[
\int_\R \E^{-2g(\lam)} d\rho(\lam) = \int_{[0,\infty)} \E^{-2g(\lam)} dR(\lam) \le \int_{[0,\infty)} \E^{-R(\lam)} dR(\lam) 
\le \int_0^\infty \E^{-\lam} d\lam< \infty,
\]
where the last step follows from the substitution rule \cite[Cor.~5.4]{tsrlsi} since $\E^{-\lam}$ is decreasing.

Now let some $\hat{g}$ be given. It remains to verify the integral representation \eqref{Mir}.
Abbreviate $\hat{g}$ times the integral on the right-hand side of \eqref{Mir} by $\ti{M}(z)$. Since $\ti{M}(z)$ is holomorphic for $z\in\C\setminus\R$
it suffices to show that $M(z)-\ti{M}(z)$ is holomorphic near any point $\lam_0\in\R$. So fix $\lam_0\in\R$ and
choose some real-valued $f\in L^2_c(a,b)$ such that $\hat{f}(z)$ defined in \eqref{defhatf} does not vanish at $\lam_0$.
Then by virtue of Lemma~\ref{lemmfM} we obtain
\begin{align*}
&M(z)-\ti{M}(z) = -\frac{E_f(z)}{\hat{f}(z)^2} + \frac{m_f(z)}{\hat{f}(z)^2} - \ti{M}(z)\\
&\quad = -\frac{E_f(z)}{\hat{f}(z)^2} + \frac{1}{\hat{f}(z)^2} \int_{\R\setminus I} \frac{d\mu_f(\lam)}{\lam-z} -
\hat{g}(z) \int_{\R\setminus I} \left(\frac{1}{\lam-z} - \frac{\lam}{1+\lam^2}\right) \frac{d\rho(\lam)}{\hat{g}(\lam)}\\
&\qquad + \hat{g}(z)\int_I \frac{\lam}{1+\lam^2}\frac{d\rho(\lam)}{\hat{g}(\lam)} +
\int_I \frac{1}{\lam-z}\left(\frac{\hat{f}(\lam)^2}{\hat{f}(z)^2} - \frac{\hat{g}(z)}{\hat{g}(\lam)}\right) d\rho(\lam),
\end{align*}
where $I$ is some small interval containing $\lam_0$ such that $\hat{f}(z)$ does not vanish in a neighborhood of $I$.
Now observe that all terms in the above representation are holomorphic near $\lam_0$. In fact, for the first four
terms this is clear and concerning the last one note that the integrand is holomorphic as a function of both
variables in a neighborhood of $(\lam_0,\lam_0)$.
\end{proof}

Choosing $\hat{g}(z)=\E^{2g(z)}$ in the previous theorem we have
\be
M(z) = \ti{E}(z) + \E^{2g(z)} \int_\R \frac{\E^{-2g(\lam)}d\rho(\lam)}{\lam-z}, \quad
\ti{E}(z) = E(z) - \E^{2g(z)} \int_\R \frac{\lam \E^{-2g(\lam)}d\rho(\lam)}{1+\lam^2}.
\ee
Hence if we choose $f(z) = -\exp(-g(z)) \ti{E}(z)$ and switch to a new system of solutions as in Remark~\ref{rem:uniq},
then, by Remark~\ref{rem:uniqrho}, the new measure is $d\ti{\rho}(\lam)= \E^{-2g(\lam)}d\rho(\lam)$, which is a finite
measure. In particular, we see that the new singular Weyl function is a Herglotz--Nevanlinna function.

\begin{corollary}\label{cor:herg}
There is always a system of real entire solutions $\ti{\theta}(z,x)$ and $\ti{\phi}(z,x)$ such that the associated
spectral measure $\ti{\rho}$ is finite and the associated singular Weyl function is a Herglotz--Nevanlinna function
given by
\be
\ti{M}(z) = \int_\R \frac{d\ti{\rho}(\lam)}{\lam-z}.
\ee
\end{corollary}

We will illustrate this result in Remark~\ref{rem:tirhobes}.

In some cases one might not want to rescale the measure too much. In such a situation the entire function
$\hat{g}(z)= \E^{-2g(z)}$ constructed in the proof of the previous theorem will usually not be optimal. 
In order to find a better $\hat{g}(z)$ note that $\phi(\lam,x)^2 + \phi'(\lam,x)^2$
is positive for $\lam\in\R$ and in $L^1(\R, (1+\lam^2)^{-1}d\rho)$ by Lemma~\ref{lemUub}.
We will exploit this observation to find a better $\hat{g}(z)$ for a large class of interesting examples in Corollary~\ref{corIntR}.

As another consequence we get a criterion when our singular Weyl $m$-function is a
generalized Nevanlinna function with no nonreal poles and the only generalized pole of nonpositive type at $\infty$.
We will denote the set of all such generalized Nevanlinna functions by $N_\kappa^\infty$ and refer to Appendix~\ref{app:gnf}
for further information.

\begin{theorem}\label{thm:nkap}
Fix the solution $\phi(z,x)$. Then there is a corresponding solution $\theta(z,x)$ such that $M(z)\in N_\kappa^\infty$
for some $\kappa\le k$ if and only if $(1+\lam^2)^{-k-1} \in L^1(\R,d\rho)$. Moreover, $\kappa=k$ if $k=0$ or
$(1+\lam^2)^{-k} \not\in L^1(\R,d\rho)$.
\end{theorem}

\begin{proof}
If $(1+\lam^2)^{-k-1} \in L^1(\R,d\rho)$ we can choose $\hat{g}(z)= (1+z^2)^k$ and by Theorem~\ref{IntR}
we have
\[
M(z) = f(z) + (1+z^2)^k \int_\R \left(\frac{1}{\lam-z} - \frac{\lam}{1+\lam^2}\right) \frac{d\rho(\lam)}{(1+\lam^2)^k},
\]
where $f(z)$ is entire. Hence choosing $\ti{\theta}(z,x) = \theta(z,x) - f(z) \phi(z,x)$ (cf.\ Remark~\ref{rem:uniq})
the corresponding Weyl $m$-function
\be
\ti{M}(z) = (1+z^2)^k \int_\R \left(\frac{1}{\lam-z} - \frac{\lam}{1+\lam^2}\right) \frac{d\rho(\lam)}{(1+\lam^2)^k}
\ee
is of the required type.

Conversely, a generalized Nevanlinna function from $N_\kappa^\infty$ admits the integral representation
\eqref{minkappa}--\eqref{minkappa'}, where the measure $d\rho$ coincides with the one from Lemma~\ref{lemrho}.
\end{proof}

Note that the condition $(1+\lam^2)^{-k-1} \in L^1(\R,d\rho)$ is related to the growth of $M(z)$ along the imaginary axis
(cf.\ Lemma~\ref{lem:gnf:grow}). In order to identify possible values of $k$ one can try to bound $\lam^{-k}$ by a
linear combination of $\phi(\lam,x)^2$ and $\phi'(\lam,x)^2$ which are in $L^1(\R, (1+\lam^2)^{-1}d\rho)$ by
Lemma~\ref{lemUub} as pointed out before.

\section{An example}
\label{sec:ex}

Our prototypical example will be the spherical Schr\"odinger equation given by
\be
H_l = - \frac{d^2}{dx^2} + \frac{l(l+1)}{x^2}, \qquad x\in(0,b), \:  l \ge -\frac{1}{2},
\ee
with the usual boundary condition at $x=0$ (for $l\in[-\frac{1}{2},\frac{1}{2})$)
\be
\lim_{x\to0} x^l ( (l+1)f(x) - x f'(x))=0,
\ee
which arises when investigating the free Schr\"odinger operator in spherical coordinates (e.g., \cite{tschroe}, \cite{wdln})
Note that we explicitly allow noninteger values of $l$ such that we also cover the case of
arbitrary space dimension $n\ge 2$, where $l(l+1)$ has to be replaced by $l(l+n-2) + (n-1)(n-3)/4$ \cite[Sec.~17.F]{wdln}.
The boundary condition used here corresponds to the Friedrich's extension and we refer to \cite{bg} for a characterization
of all possible boundary conditions in terms of Rellich's {\em Anfangszahlen} (see also \cite{ek}).

Two linearly independent solutions of
\be
- u''(x) + \frac{l(l+1)}{x^2} u(x) = z u(x)
\ee
are given by
\be\label{defphil}
\phi_l(z,x) = z^{-\frac{2l+1}{4}} \sqrt{\frac{\pi x}{2}} J_{l+\frac{1}{2}}(\sqrt{z} x),
\ee
\be\label{defthetal}
\theta_l(z,x) = -z^{\frac{2l+1}{4}} \sqrt{\frac{\pi x}{2}} \begin{cases}
\frac{-1}{\sin((l+\frac{1}{2})\pi)} J_{-l-\frac{1}{2}}(\sqrt{z} x), & {l+\frac{1}{2}} \in \R_+\setminus \N_0,\\
Y_{l+\frac{1}{2}}(\sqrt{z} x) -\frac{1}{\pi}\log(z) J_{l+\frac{1}{2}}(\sqrt{z} x), & {l+\frac{1}{2}} \in\N_0,\end{cases}
\ee
where $J_{l+\frac{1}{2}}$ and $Y_{l+\frac{1}{2}}$ are the usual Bessel and Neumann functions \cite{as}.
All branch cuts are chosen along the negative real axis unless explicitly stated otherwise.
If $l$ is an integer they of course reduce to spherical Bessel and Neumann functions
and can be expressed in terms of trigonometric functions (cf.\ e.g.\ \cite[Sect.~10.4]{tschroe})

In particular, both functions are entire and their Wronskian is given by
\be
W(\theta_l(z),\phi_l(z))=1.
\ee
Moreover, on $(0,\infty)$ and $l\in\N_0$ we have
\begin{align}\nn
u_+(z,x) &= \theta_l(z,x) + M_l(z) \phi_l(z,x)\\
&= - (\I\sqrt{-z})^{l+\frac{1}{2}} \sqrt{\frac{\pi x}{2}} \left( Y_{l+\frac{1}{2}}(\I \sqrt{-z} x) - \I J_{l+\frac{1}{2}}(\I\sqrt{-z} x)\right)
\end{align}
with
\be
M_l(z) = \begin{cases}
\frac{-1}{\sin((l+\frac{1}{2})\pi)} (-z)^{l+\frac{1}{2}}, & {l+\frac{1}{2}}\in\R_+\setminus \N_0,\\
\frac{-1}{\pi} z^{l+\frac{1}{2}}\log(-z), & {l+\frac{1}{2}} \in\N_0,\end{cases}
\ee
where all branch cuts are chosen along the negative real axis. The associated spectral measure
is given by
\be
d\rho_l(\lam) = \chi_{[0,\infty)}(\lam) \lam^{l+\frac{1}{2}} \frac{d\lam}{\pi},  \qquad l \geq -\frac{1}{2},
\ee
and the associated spectral transformation is just the usual Hankel transform. Furthermore, one
infers that $M_l(z)$ is in the generalized Nevanlinna class $N_\kappa^\infty$ with $\kappa=\floor{l/2 + 3/4}$.
Here $\floor{x}= \max \{ n \in \Z | n \leq x\}$ is the usual floor function.

For more information we refer to Section~4 of \cite{gz}, to \cite{ek}, where
the limit circle case $l\in[-1/2,1/2)$ is considered, and to Section~5 of \cite{fl},
where the case $q(x) = l(l+1)/x^2 - a/ x$ is worked out (see also \cite{DSh_00}, \cite{kl}).

\begin{remark}\label{rem:tirhobes}
It is also interesting to rescale $\phi(z,x)$ and $\theta(z,x)$ according to Corollary~\ref{cor:herg} in order to obtain
a Herglotz--Nevanlinna function. Clearly, $\E^{-\lam}\in L^1(\R_+, d\rho_l)$. Then setting
\be
d\widetilde{\rho}_l(\lam):=\E^{-\lam}d\rho_l(\lam)
\ee
and using \cite[formula (3.383.10)]{GR} we obtain for the corresponding Herglotz--Nevanlinna function
\be
\ti{M}(z) = \int_0^\infty \frac{d\ti{\rho}_l(\lam)}{\lam-z}  = \frac{\Gamma(l+\frac{3}{2})}{\pi} (-z)^{l+\frac{1}{2}} \E^{-z} \Gamma(-l-\frac{1}{2},-z),
\ee
where $\Gamma(\alpha,z)$ is the incomplete Gamma function. As discussed before Corollary~\ref{cor:herg} $\ti{M}(z)$
is the singular Weyl function for $H_l$ corresponding to the fundamental system of solutions given by
\be
\ti{\phi}_l(z,x)=\E^{-\frac{z}{2}}\phi_l(z,x),\qquad \ti{\theta}_l(z,x)=\E^{\frac{z}{2}}\big(\theta_l(z,x) - E_l(z)\phi_l(z,x)\big),
\ee
where
\be
E_l(z) = \begin{cases} \frac{1}{\pi}\Gamma(l+\frac{3}{2}) \sum_{j=0}^\infty \frac{z^j}{j!(j-l-1/2)}, & {l+\frac{1}{2}}\in\R_+\setminus \N_0,\\
-\frac{1}{\pi}z^{l+\frac{1}{2}}\big(\log(-z)+\Gamma(0,-z)+\E^{-z}\sum_{j=1}^{l+1/2}\frac{j!}{z^{j}}\big), & {l+\frac{1}{2}} \in\N_0.\end{cases}
\ee
To check this use 8.354.2 and 8.334.3 from \cite{GR} in the noninteger case $l+\frac{1}{2}\notin \N_0$ and
8.352.5 \cite{GR} in the integer case $l+\frac{1}{2}\in \N_0$.
\end{remark}

\section{Exponential growth rates}

While the system $\phi(z,x)$, $\theta(z,x)$ found in Section~\ref{sec:swm} was good enough to define a singular
Weyl $m$-function and an associated spectral measure, it does not suffice for the proof of our local Borg--Marchenko
theorem. For this we will need information on the exponential growth rate of the solutions $\phi(z,x)$ and $\theta(z,x)$.
To this end we introduce the following additional assumption.

\begin{hypothesis}\label{hyp:ev}
Suppose $a$ is finite.
Suppose that the spectrum of $H^X_{(a,c)}$, $X\in\{D,N\}$ is purely discrete
\be
\sig(H^D_{(a,c)}) = \sig_d(H^D_{(a,c)}) = \{ \mu_j(c) \}_{j\in\N}, \quad
\sig(H^N_{(a,c)}) = \sig_d(H^N_{(a,c)}) = \{ \nu_j(c) \}_{j\in\N_0}
\ee
and bounded below. In addition, suppose one of the following conditions holds:
\begin{enumerate}
\item
For all values $s> \frac{1}{2}$ we have
\be\label{estorder}
\sum_{j\in\N} \frac{1}{1+|\mu_j(c)|^s}<\infty.
\ee
\item[(i')]
Assume that there exists a $\Delta(c)>0$ such that
\be
\mu_j(c) = \left(\frac{\Delta(c)}{\pi} j\right)^2 + o(j^2).
\ee
\item
For sufficiently large $j$ we have
\be\label{estmunu}
\mu_j(c)-\nu_{j-1}(c) \ge \frac{1}{\nu_{j-1}(c)^r} + \frac{1}{\mu_j(c)^r}, \qquad
\nu_j(c)-\mu_j(c) \ge \frac{1}{\nu_j(c)^r} + \frac{1}{\mu_j(c)^r},
\ee
for some $r>0$.
\end{enumerate}
\end{hypothesis}

\begin{remark}
Clearly (i) implies that the resolvent $(H^X_{(a,c)}-z)^{-1}$ is trace class and (i') implies (i).

The condition (ii) requires that the Dirichlet and Neumann spectra do not get too close to each other.
Recall that both spectra are interlacing:
\be\label{interlacing}
\nu_0(c)<\mu_1(c)<\nu_1(c)<\mu_2(c)<\cdots.
\ee

Condition (ii) is also reasonable, since by virtue of Krein's theorem \cite[Thm.~27.2.1]{lev} we have
\be\label{mmkrein}
m_-(z) = C (z-\nu_0) \prod_{j\in\N} \frac{1 - \frac{z}{\nu_j(c)}}{1 - \frac{z}{\mu_j(c)}} =
- C' \prod_{j\in\N} \frac{1 - \frac{z}{\nu_{j-1}(c)}}{1 - \frac{z}{\mu_j(c)}}, \quad C,C'>0.
\ee
However, if $\nu_j(c)$ and $\mu_j(c)$ or  $\nu_{j-1}(c)$ and $\mu_j(c)$ are eventually too close,
this would contradict the well-known asymptotics
\be
m_-(z) = -\sqrt{-z} + O(1),
\ee
where the branch cut is chosen along the negative real axis.

Finally, we note that the assumption that $H$ is bounded from below could be dropped
in principle. However, for negative potentials the convergence exponent
(i.e., the infimum over all $s$ for which \eqref{estorder} holds) can be larger than $\frac{1}{2}$
(see e.g.\ \cite{bk}) and hence one has to work with higher-order Hadamard products below.
In this respect one should also mention a result by Krein \cite{kr} which states that if $\tau$ is
regular at $a=0$ and limit circle at $b=\infty$, then the growth order of $\phi(z,x)$ is in $[1/2,1]$.
Moreover, if it is equal to $1/2$ (resp.\ $1$), then $\phi(z,x)$ is of maximal (resp.\ minimal) type.
\end{remark}

It will be convenient to look at the eigenvalue counting function
\be
N_{H^X_{(a,c)}}(R) = \#\{\lam \in\sig(H^X_{(a,c)}) | \lam \le R \} = \dim\ran P_{H^X_{(a,c)}}((-\infty,R]),
\ee
where $P_{H^X_{(a,c)}}(\Omega)$ is the family of spectral projections associated with the self-adjoint
operator ${H^X_{(a,c)}}$. The connection with our hypothesis is given by the fact that the convergence
exponent of $\mu_j(c)$ (i.e., the infimum over all $s$ for which \eqref{estorder} holds) is equal to the order of the
counting function (\cite[Lem.~3.2.2]{lev}):
\be
\inf \Big\{ s \Big| \sum_{j\in\N} \frac{1}{1+|\mu_j(c)|^s}<\infty \Big\} =\limsup_{R\to\infty} \frac{\log(N_{H^D_{(a,c)}}(R))}{\log(R)}.
\ee
Condition (i') of Hypothesis~\eqref{hyp:ev} is equivalent to
\be\label{eqwa}
N_{H^D_{(a,c)}}(R) = \frac{\Delta(c)}{\pi} R^{1/2} + o(R^{1/2}).
\ee
Moreover, since changing the boundary conditions (at one endpoint) amounts
to a rank-one resolvent perturbation, the corresponding eigenvalue counting functions satisfy
$|N_{H^D_{(a,c)}}(R) - N_{H^N_{(a,c)}}(R)| \le 1$ (see e.g.\ \cite[Corollary~II.2.1]{gk}). Hence
since we will eventually only be interested in the growth rates of these functions it will suffice
to look at $H^D_{(a,c)}$.

Finally, recall that if $H^D_{(a,c)}$ is regular we have the well-known Weyl asymptotics
\be\label{weylreg}
N_{H^D_{(a,c)}}(R) = \frac{c-a}{\pi} R^{1/2} + o(R^{1/2}),
\ee
that is, \eqref{eqwa} holds with $\Delta(c)=c-a$.

\begin{lemma}\label{lem:hic}
The first condition (i) in Hypothesis~\ref{hyp:ev} is independent of the value of $c$
and holds for $H^D_{(a,c)}$ if and only if it holds for $H^N_{(a,c)}$.
Moreover, condition (i) cannot hold for any $s \le \frac{1}{2}$.

Similarly, condition (i') in Hypothesis~\ref{hyp:ev} is independent of the value of $c$
and holds for $H^D_{(a,c)}$ if and only it holds for $H^N_{(a,c)}$ (with the same $\Delta(c)$).
Moreover,
\be
\Delta(x) = \Delta_0 + x -a.
\ee
\end{lemma}

\begin{proof}
As already pointed out, the boundary condition is irrelevant.
To see that (i) is independent of $c$ it suffices to consider the operator $\tilde{H}_{(a,c+\delta)}^D:=H^D_{(a,c)}\oplus H^D_{(c,c+\delta)}$
which differs from $H^D_{(a,c+\delta)}$ by a rank one resolvent perturbation. Hence
\[
N_{\ti{H}^D_{(a,c+\delta)}}(R) = N_{H^D_{(a,c)}}(R) + N_{H^D_{(c,c+\delta)}}(R) = N_{H^D_{(a,c)}}(R) + \frac{\delta}{\pi} R^{1/2} + o(R^{1/2})
\]
by \eqref{weylreg} since the points $c$ and $c+\delta$ are regular.

Similarly, the last claim of the first part follows using $\tilde{H}_{(a,c)}^D:=H^D_{(a,a+\eps)}\oplus H^D_{(a+\eps,c)}$ and
\[
N_{\ti{H}^D_{(a,c)}}(R) \ge N_{H^D_{(a+\eps,c)}}(R) = \frac{c-a-\eps}{\pi} R^{1/2} + o(R^{1/2})
\]
since if \eqref{estorder} holds for $s=\frac{1}{2}$, then $\int_1^\infty R^{-3/2} N_{H^D_{(a,c)}}(R) dR< \infty$ by \cite[Lem.~3.2.1]{lev}.

The second part follows analogously.
\end{proof}

Now we are ready to draw some first conclusions for $\phi(z,x)$.

\begin{lemma}\label{lem:phiev}
Suppose Hypothesis~\ref{hyp:ev} (i) holds.
Then there is a real entire solution $\phi(z,x)$ of growth order $\frac{1}{2}$. This solution is unique
up to a constant and given by
\be\label{phiprod}
\phi(z,x) =  \phi(0,x) \prod_{j\in\N} \Big(1-\frac{z}{\mu_j(x)}\Big)
\ee
(if $\mu_{j_0}(x)=0$ one has to replace $\phi(0,x)(1-z/\mu_{j_0}(x))$ by $\dot{\phi}(0,x) z$, where
the dot indicates a derivative with respect to $z$).
Furthermore,
\be\label{phias}
\phi(z,x) = \phi(z,x_0) \E^{-(x-x_0) \sqrt{-z}} (1 + O(1/\sqrt{-z})), \qquad x_0,x \in (a,b),
\ee
as $|z|\to\infty$ along any nonreal ray.
\end{lemma}

\begin{proof}
By Lemma~\ref{lem:phi} we already know that there is an entire solution $\phi(z,x)$. Moreover,
choosing
\be
\alpha(z) = \prod_{j\in\N} \Big(1-\frac{z}{\mu_j(c)}\Big), \qquad \beta(z) = -C'\prod_{j\in\N_0} \Big(1-\frac{z}{\nu_j(c)}\Big)
\ee
in \eqref{phiab}, where $C'$ is chosen such that \eqref{mmkrein} holds, we see that there is an entire solution of growth order at
most $1/2$ since $\alpha(z)$, $\beta(z)$ are of growth order $1/2$ by Borel's theorem (\cite[Thm.~4.3.3]{lev}) and the same is true
for $c(z,x)$, $s(z,x)$ by \cite[Lem.~9.18]{tschroe}. Moreover, any solution of growth order at most $1/2$ is uniquely determined
by its zeros and its value (respectively its derivative, if $\phi(0,x)=0$) at $0$ by the Hadamard product theorem
\cite[Thm.~4.2.1]{lev}. But the zeros of $\phi(z,x)$ are precisely the eigenvalues $\mu_j(x)$ of $H^D_{(a,x)}$ and all of
them are simple since the Herglotz--Nevanlinna function $m_-(z,x)= \phi'(z,x) / \phi(z,x)$ has a first-order pole at
every eigenvalue. Thus any entire solution $\phi(z,x)$ of growth order at most $1/2$ is given by \eqref{phiprod}. By
Lemma~\ref{lem:hic} and Borel's theorem the growth order of \eqref{phiprod} is equal to $1/2$.

Finally,
\[
\phi(z,x) = \phi(z,c) ( c(z,x) - m_-(z) s(z,x) )
\]
and using the well-known asymptotics of $c(z,x)$, $s(z,x)$, and $m_-(z)$ (cf.\ \cite[Lem.~9.18 and Lem.~9.19]{tschroe})
we see \eqref{phias} for $x_0=c$ and $x>x_0$. The case $x<x_0$ follows after reversing the roles of $x_0$ and $x$.
Since $c$ is arbitrary the proof is complete.
\end{proof}

Note that we also have
\be\label{phipprod}
\phi'(z,x) =  \phi'(0,x) \prod_{j\in\N} \Big(1-\frac{z}{\nu_j(x)}\Big)
\ee
(if $\nu_{j_0}(x)=0$ one has to replace $\phi'(0,x)(1-z/\nu_{j_0}(x))$ by $\dot{\phi}'(0,x) z$).

\begin{lemma}\label{lem:phiev2}
Suppose Hypothesis~\ref{hyp:ev} (i') holds. Then
\be\label{phias2}
\phi(z,x) = \E^{-(\Delta_0+x-a) \sqrt{-z} + o(z/\im(\sqrt{-z}))}
\ee
as $|z|\to\infty$ along any ray except the positive real axis.
\end{lemma}

\begin{proof}
The claim follows from \eqref{eqwa} and \cite[Thm.~12.1.1]{lev}.
\end{proof}

Next, we establish a practical condition which ensures that the first part of our hypothesis holds.

\begin{lemma}\label{lem:wa}
Suppose $a\in\R$ and $q(x)$ satisfies
\be\label{eq:knez}
\min(q(x),0) \ge -\frac{C}{(x-a)^2} - Q(x)
\ee
with $C<\frac{1}{4}$ and $(x-a) Q(x) \in L^1(a,c)$ or $C=\frac{1}{4}$ and $(x-a) (1-\log(x-a)) Q(x) \in L^1(a,c)$.

Then Hypothesis~\ref{hyp:ev} (i') holds with $\Delta(x)=x-a$.
\end{lemma}

\begin{proof}
Denote the right-hand side of \eqref{eq:knez} by $\hat{q}(x):=-C(x-a)^2-Q(x)$.
First of all note that the operator $\hat{H}$ associated with $\hat{q}$ has purely discrete spectrum and is bounded from below,
by \cite[Thm.~2.4]{kst} (see also Section~\ref{sec:ap}). In particular,  the associated differential equation $\hat{\tau}-\lam$
is nonoscillatory for every $\lam\in\R$. Hence the same is true for $\tau-\lam$ by Sturm's comparison theorem and thus $H$
has no essential spectrum and is bounded from below (cf.\ \cite[Thm.~14.9]{wdln}).

Next, we can again replace  $H^D_{(a,c)}$ by $\tilde{H}_{(a,c)}^D:=H^D_{(a,a+\eps)}\oplus H^D_{(a+\eps,c)}$ for
any $\eps\in(0,c-a)$ since both operators differ by a rank one resolvent perturbation.
Denote by $\ti{N}_{(a,c)}(R)$, $N_{(a,a+\eps)}(R)$, and $N_{(a+\eps,c)}(R)$ the corresponding eigenvalue
counting functions. Then,
\[
\ti{N}_{(a,c)}(R) = N_{(a,a+\eps)}(R) + N_{(a+\eps,c)}(R) = \frac{c-a}{\pi} R^{1/2} + N_\eps(R) + o(R^{1/2})
\]
with
\[
N_\eps(R)= N_{(a,a+\eps)}(R) - \frac{\eps}{\pi} R^{1/2}
\]
by \eqref{weylreg} since the points $a+\eps$ and $c$ are regular. Next, note that the eigenvalue counting function associated with
$\hat{H}^D_{(a,a+\eps)}$ satisfies
\[
\hat{N}_{(a,a+\eps)}(R) =  \frac{\eps}{\pi} R^{1/2} + o(R^{1/2})
\]
by \cite[Thm.~2.5]{kst} (see also Section \ref{sec:ap}). Since $q(x) \ge \hat{q}(x)$ by assumption, the min--max principle implies
\[
N_{(a,a+\eps)}(R) \leq \hat{N}_{(a,a+\eps)}(R)
\]
and we obtain
\begin{align*}
-\frac{\eps}{\pi} R^{1/2} \le N_\eps(R) \le o(R^{1/2})
\end{align*}
for any $\eps\in(0,c-a)$. Hence $N_\eps(R) = o(R^{1/2})$ and \eqref{eqwa} holds with $\Delta(x)=x-a$. This is equivalent to (i') as noted earlier.
\end{proof}

\begin{remark}
This condition is not far from being optimal since any potential $q(x)$ satisfying
\[
q(x) \le -\frac{C}{(x-a)^2}
\]
for some $C>\frac{1}{4}$ fails to be bounded from below. To see this observe that
the potential $-C (x-a)^{-2}$ is oscillatory near $a$ and thus the same is true
for any potential below it by Sturm's comparison theorem.
\end{remark}

In order to prove our next lemma we will need the following version of the corona theorem for entire
functions:

\begin{theorem}[\cite{hor}]\label{thm:cor}
Let $R_s(\C)$, $s> 0$, be the ring of all entire functions $f(z)$ for which there are constants $A,B >0$ such that
\be
|f(z)| \le B \exp(A |z|^s).
\ee
Then $f_j(z)$, $j=1,\dots,n$ generate $R_s(\C)$ if and only if
\be
|f_1(z)| + \cdots + |f_n(z)| \ge b \exp(-a |z|^s)
\ee
for some constants $a,b>0$.
\end{theorem}

\begin{proof}
This follows from Theorem~1 of \cite{hor} upon choosing $\Omega=\C$, $p(z)=|z|^s$,
$K_1=K_2=0$, $K_2=K_4=\max(1,2^{s-1})$.
\end{proof}

Moreover, we will need the following technical estimate for Hadamard products.

\begin{lemma}\label{lemesthp}
Let $\{\zeta_j\}_{j=1}^\infty \subseteq \C\backslash\{0\}$ be a sequence with no finite accumulation point
and with growth exponent $\rho<1$. Let $r>0$. Then for every $s$ with $\rho< s < 1$ there are
constants $A$ and $B$ (depending on $s$ and $r$) such that the Hadamard product satisfies
\begin{align}\label{estcp}
\Big|\prod_{j=1}^\infty \Big(1-\frac{z}{\zeta_j}\Big) \Big| \ge B \exp(-A |z|^s),
\end{align}
except possibly when $z$ belongs to one of the discs $|z-\zeta_j|<|\zeta_j|^{-r}$.
\end{lemma}

Arguments like this are well known (see also \cite[Sect.~12]{lev}) but we could not find a reference for the particular version required here
and hence we provide a proof for the sake of completeness.

\begin{proof}
We split
\[
\prod_{j=1}^\infty \Big(1-\frac{z}{\zeta_j}\Big) = \prod_{|\zeta_j| > 2|z|} \Big(1-\frac{z}{\zeta_j}\Big)
\prod_{|\zeta_j| \le 2|z|} \Big(1-\frac{z}{\zeta_j}\Big)
\]
By virtue of $|1-z| \geq \E^{-c|z|}$ for $2|z| \le 1$ (with $c = 2\log(2)$), the first product can be estimated by
\[
\left| \prod_{|\zeta_j| > 2|z|} \Big(1-\frac{z}{\zeta_j}\Big) \right| \ge
\prod_{|\zeta_j| > 2|z|} \E^{-c|z/\zeta_j|}  \ge
\E^{-c 2^{s-1} |z|^s \sum_j |\zeta_j|^{-s}}
\]
as required. For the second product we use our requirement $|z-\zeta_j| \ge |\zeta_j|^{-r}$ for all $j$
to obtain
\[
\left| \prod_{|\zeta_j| \le 2|z|} \Big(1-\frac{z}{\zeta_j}\Big) \right| =
\prod_{|\zeta_j| \le 2|z|} \big| \frac{z-\zeta_j}{\zeta_j} \Big| \ge
\prod_{|\zeta_j| \le 2|z|} |\zeta_j|^{-r-1}.
\]
Finally, we have
\[
(r+1) \sum_{|\zeta_j| \le 2|z|} \log|\zeta_j| \le (r+1) n(2|z|) \log(2|z|) \le C |z|^s \log|z|
\]
for $|z| \ge 2$, where $n(R) = \#\{\zeta_j | |\zeta_j|\le R\}$ is the counting function for zeros. Hence
the claim holds for any $s'>s$ and we are done.
\end{proof}

Now we are ready to show the following:

\begin{lemma}\label{lem:cor}
Suppose Hypothesis~\ref{hyp:ev} (i) and (ii) hold. Then for every $s>\frac{1}{2}$ there is a second solution as
in Lemma~\ref{lem:theta} with $\gamma(z), \delta(z)\in R_s(\C)$.
In particular, we have $\theta(.,x) \in R_s(\C)$ in this case.
\end{lemma}

\begin{proof}
We will fix some $s> \frac{1}{2}$ such that $\alpha(z), \beta(z) \in R_s(\C)$. Using \eqref{thetagd} the Wronski condition
implies that we need to find $\gamma(z), \delta(z) \in R_s(\C)$ such that
\be\label{corona_equation}
\gamma(z) \beta(z) - \delta(z) \alpha(z) = 1.
\ee
This follows from Theorem~\ref{thm:cor} once we show
\be\label{corona_cond}
|\alpha(z)|+|\beta(z)| \ge B \exp(-A |z|^s)
\ee
for some constants $A, B>0$, which depend only on $s$ and $r$.

By \eqref{interlacing} and Hypothesis \ref{hyp:ev} there is an $N\in \N$ such that the disks
$|z-\mu_j|<|\mu_j|^{-r}$ and $|z-\nu_i|<|\nu_i|^{-r}$ are disjoint for all $j,i \ge N$. Hence Lemma~\ref{lemesthp}
implies that we have \eqref{estcp} for either $\alpha(z)$ or $\beta(z)$ and hence in particular for the sum of
both which proves the desired inequality \eqref{corona_cond}.
If $\gamma(z)$, $\delta(z)$ are not real, $\frac{1}{2}(\gamma(z) + \gamma(z^*)^*)$, $\frac{1}{2}(\delta(z) + \delta(z^*)^*)$
will be.
\end{proof}

As a consequence of the proof we also obtain the following:

\begin{corollary}\label{corIntR}
Suppose Hypothesis~\ref{hyp:ev} (i) and (ii) hold. Then one can use $\hat{g}(z)=\exp(z^2)$ in Theorem~\ref{IntR}.
If in addition $H$ is bounded from below, then one can also use $\hat{g}(z)=\exp(z)$.
\end{corollary}

\begin{proof}
Note that from the proof of the previous theorem we know $|\phi(z,c)|^2 + |\phi'(z,c)|^2 \ge \frac{1}{2}B^2 \exp(-2A |z|^s)$
for any $s>1/2$ (cf.\ \eqref{corona_cond} and recall \eqref{eqalbe}).
Since this function is in $L^1(\R,(1+\lam^2)^{-1}d\rho)$ by Lemma~\ref{lemUub}, the claim follows.
\end{proof}

\begin{remark}\label{rem:uniq2}
While under Hypothesis~\ref{hyp:ev} (i) a solution $\phi(z,x)$ of growth order at most $\frac{1}{2}$ is unique up to a constant,
$\theta(z,x)$ is only unique up to $f(z) \phi(z,x)$, where $f(z)$ is an entire function of growth order at most
$\frac{1}{2}$.
\end{remark}

\section{A local Borg--Marchenko uniqueness result}

As in the previous sections we assume that we have an entire system of solutions
$\theta(z,x)$ and $\phi(z,x)$ satisfying $W(\theta(z),\phi(z))=1$.

\begin{lemma}
The singular Weyl $m$-function $M(z)$ and the solution $\psi(z,x)$ defined in \eqref{defpsi}
have the following asymptotics
\begin{align}\label{asymM}
M(z) &= -\frac{\theta(z,x)}{\phi(z,x)} + O(\frac{1}{\sqrt{-z}\phi(z,x)^2}),\\ \label{asympsi}
\psi(z,x) &= \frac{1}{2\sqrt{-z} \phi(z,x)} \left( 1 + O(\frac{1}{\sqrt{-z}}) \right)
\end{align}
as $|z|\to\infty$ in any sector $|\re(z)| \le C |\im(z)|$.
\end{lemma}

\begin{proof}
This follows by solving the well-known asymptotical formula (e.g., combine Lem.~9.19 and (9.101), (9.103) from \cite{tschroe})
for the diagonal of Green's function
\[
G(z,x,x) = \phi(z,x) \psi(z,x) = \frac{1}{2\sqrt{-z}} + O(z^{-1})
\]
for $\psi(z,x)$ and $M(z)$.
\end{proof}

In particular, \eqref{asymM} shows that asymptotics of $M(z)$ immediately follow once one
has corresponding asymptotics for $\theta(z,x)$ and $\phi(z,x)$. Moreover, the leading
asymptotics up to $O(z^{-1/2}\phi(z,x)^{-2})$ depends only on the values of $q(x)$ for
$x\in (a,c)$ (and on the choice of $\phi(z,x)$ and $\theta(z,x)$).
The converse is the content of the following Borg--Marchenko type uniqueness result.

\begin{theorem}\label{thmbm}
Let $q_0$ and $q_1$ be two potentials satisfying Hypothesis~\ref{hyp:ev} (i) and (ii).
Let $c\in(a,b)$ and $\phi_j(z,x), \theta_j(z,x) \in R_s(\C)$, $j=0,1$, for some $s \ge \frac{1}{2}$ be given.

Then, if  $\frac{\phi_1(z,x)}{\phi_0(z,x)} =1 +o(1)$ for one (and hence by \eqref{phias} for all) $x\in(a,c)$ and
$M_1(z)-M_0(z) = f(z) + O(\phi_0(z,c-\eps)^{-2})$ for some $f(z)\in R_s(\C)$ and
for every $\eps>0$ as $z\to\infty$ along some nonreal rays dissecting the complex plane into sectors of length less
than $\pi/s$, we have $q_0(x)=q_1(x)$ for $x\in(a,c)$.
\end{theorem}

\begin{proof}
First of all note that after the replacement $\theta_0(z,x) \to \theta_0(z,x) - f(z) \phi_0(z,x)$ (cf.\ Remark~\ref{rem:uniq})
we can assume $f(z)=0$ without loss of generality.

Now note that the entire function
\begin{align*}
\phi_1(z,x) \theta_0(z,x) - \phi_0(z,x) \theta_1(z,x) =& \phi_1(z,x) \psi_0(z,x) - \phi_0(z,x) \psi_1(z,x)\\
& +(M_1(z)-M_0(z)) \phi_0(z,x) \phi_1(z,x)
\end{align*}
vanishes as $z\to\infty$ along our nonreal rays for fixed $x\in(a,c)$. For the first two terms this
follows from \eqref{asympsi} together with our hypothesis that $\phi_0(z,x)$ and $\phi_1(z,x)$
have the same asymptotics. For the last term this follows by our assumption on $M_1(z)-M_0(z)$.
Moreover, by our hypothesis this function has an order of growth $s$ and thus we
can apply the Phragm\'en--Lindel\"of theorem (e.g., \cite[Sect.~6.1]{lev}) in the angles bounded by our rays.
This shows that the function is bounded on all of $\C$.
By Liouville's theorem it must be constant and since it vanishes along a ray, it must be zero; that is,
$\phi_1(z,x) \theta_0(z,x) = \phi_0(z,x) \theta_1(z,x)$ for all $z\in\C$ and $x\in(a,c)$.
Dividing both sides of this identity by $\phi_0\phi_1$, differentiating with respect to $x$, and using $W(\theta_j(z),\phi_j(z))=1$ shows
$\phi_1(z,x)^2 = \phi_0(z,x)^2$. Taking the logarithmic derivative further gives
$\phi_1'(z,x)/\phi_1(z,x)=\phi_0'(z,x)/\phi_0(z,x)$ and differentiating once more shows
$\phi_1''(z,x)/\phi_1(z,x)=\phi_0''(z,x)/\phi_0(z,x)$. This finishes the proof since $q_j(x)=z + \phi_j''(z,x)/\phi_j(z,x)$.
\end{proof}

Note that while at first sight it might look like the condition $\frac{\phi_1(z,x)}{\phi_0(z,x)} =1 +o(1)$ requires knowledge
of $\phi_j(z,x)$, this is not the case, since the high energy asymptotics will only involve some qualitative information
on the kind of the singularity at $a$ as we will show in the next section. Moreover, one could even allow for
different values of $b_0$ and $b_1$ for $q_0$ and $q_1$, respectively.

Next, the appearance of the additional freedom $f(z)$ just reflects the fact that we only ensure the same normalization
for $\phi_1(z,x)$ and $\phi_2(z,x)$ but not for $\theta_1(z,x)$ and $\theta_2(z,x)$ (cf.\ Remark~\ref{rem:uniq2}).

Finally, assuming Hypothesis~\ref{hyp:ev} (i'), then by Lemma~\ref{lem:phiev2}, one can write the condition
$M_1(z)-M_0(z) = f(z) + O(\phi(z,c-\eps)^{-2})$ in the more usual form $M_1(z)-M_0(z) = f(z) + O(\E^{- 2(\Delta_0 + c-a -\eps) \sqrt{-z}})$
for every $\eps > 0$. In this respect recall that Hypothesis~\ref{hyp:ev} (i') holds under the assumptions of
Lemma~\ref{lem:wa} with $\Delta_0=0$.

\section{Applications to perturbed Bessel potentials}
\label{sec:ap}

Here we want to continue our example from Section~\ref{sec:ex} and consider
perturbations
\be
H = H_l + q, \qquad x\in(0,b),
\ee
(taking the same boundary conditions as for $H_l$ at $x=0$), where $q$ satisfies

\begin{hypothesis}\label{hyp:q}
Let $l\in [-\frac{1}{2},\infty)$. Set
\be \label{f1}
\ti{q}(x) = \begin{cases}
|q(x)|, & l > -\frac{1}{2},\\
(1-\log(x)) |q(x)|, & l = -\frac{1}{2},
\end{cases}
\ee
and suppose $q \in L^1_{loc}(0,\infty)$ is real-valued such that
\be\label{a0!}
\begin{array}{c}
x \ti{q}(x) \in L^1(0,1).
\end{array}
\ee
\end{hypothesis}

By \cite[Lem~2.2]{kst} there is an entire solution $\phi(z,x)$ of growth order $\frac{1}{2}$ which is square integrable near $0$.
In particular, it is shown that this solution satisfies
\begin{align}\label{asymphibes}
\phi(z,x) &= z^{-\frac{l+1}{2}}\left(\sin\bigl(\sqrt{z} x- \frac{l \pi}{2}\bigr) + o\bigl(\E^{x |\im(\sqrt{z})|}\bigr)\right),\\\label{asymphibes2}
\phi'(z,x) &= z^{-\frac{l}{2}}\left(\cos\bigl(\sqrt{z} x- \frac{l \pi}{2}\bigr) + o\bigl(\E^{x |\im(\sqrt{z})|}\bigr)\right),
\end{align}
as $|z|\to\infty$. As a first consequence these asymptotics imply that $M(z)$ is a generalized Nevanlinna function (cf.\ Appendix~\ref{app:gnf}).

\begin{theorem}\label{thmpbgn}
Choose $\phi(z,x)$ as in \eqref{asymphibes}.
There is a choice for $\theta(z,x)$ such that
\be
M(z) = (1+z^2)^k \int_\R \left(\frac{1}{\lam-z} - \frac{\lam}{1+\lam^2}\right) \frac{d\rho(\lam)}{(1+\lam^2)^k}, \quad k = \ceil{\frac{l+1}{2}}.
\ee
In particular, $M(z)\in N_\kappa^\infty$ for some $\kappa \le k$.
\end{theorem}

\begin{proof}
Just note that by \eqref{asymphibes} and \eqref{asymphibes2} we have $\phi(\lam,x)^2 + \lam^{-1} \phi'(\lam,x)^2 =
\lam^{-l-1} (1 + o(1))$. Since this function is in $L^1(\R, (1+\lam^2)^{-1}d\rho)$ by Lemma~\ref{lemUub}, the result follows from
Theorem~\ref{thm:nkap}.
\end{proof}

In the case of analytic perturbations this result was first obtained by \cite{fl} and \cite{kl} (the latter only considered a Coulomb term
$q(x)=q_1/x$). Moreover, the correct value of $\kappa= \floor{\frac{l}{2}+\frac{3}{4}}$ (cf.\ Section~\ref{sec:ex}) was determined, whereas
the above result gives only an estimate from above. On the other hand, our result significantly extends all previous results since it
does not require analyticity of the perturbation! Determining the correct value of $\kappa$ and an explicit construction of a
corresponding $\theta(z,x)$, without requiring analyticity of  the perturbation, will be addressed in a follow-up paper \cite{kt}
based on a detailed study of the solutions of the underlying differential equation combined with the results of the present paper.

Next we turn to our local Borg--Marchenko result.

\begin{lemma}
Suppose Hypothesis~\ref{hyp:q}. Then $H$ satisfies Hypothesis~\ref{hyp:ev} (i') and (ii) for any $r\ge 0$.
\end{lemma}

\begin{proof}
Without loss of generality choose $c=1$ in Hypothesis~\ref{hyp:ev}.
Then the claim is immediate from the eigenvalue asymptotics \cite[Thm.~2.5]{kst}
\begin{align}
\mu_j &= \left(\pi\Big(j + \frac{l}{2} \Big) + \eps_j \right)^2,\\
\nu_j &= \left(\pi\Big(j + \frac{l+1}{2} \Big) + \ti{\eps}_j \right)^2,
\end{align}
where $\eps_j, \ti{\eps}_j \to 0$.
\end{proof}

Moreover, using the detailed asymptotics from \cite[Lem~2.2]{kst} we can even strengthen the content of Lemma~\ref{lem:cor} in
this special case.

\begin{lemma}\label{lem:bes}
Suppose Hypothesis~\ref{hyp:q}. Then there is a system of real entire solutions $\phi(z,x)$ and $\theta(z,x)$ with
$W(\theta(z),\phi(z))=1$ such that $\phi(.,x), \theta(.,x) \in R_{1/2}(\C)$ for every $x>0$.
\end{lemma}

\begin{proof}
Let us take $c=1$ for notational simplicity.
For the function $\phi(z,x)$ we can take the one from \cite[Lem~2.2]{kst} (it is unique up to a constant by Remark~\ref{rem:uniq2}).
Then the estimates \eqref{asymphibes} and \eqref{asymphibes2} imply
\begin{align*}
\frac{1}{\sqrt{z}} \beta(z) - \I \alpha(z)&= \frac{1}{\sqrt{z}} \phi'(z,1) - \I \phi(z,1)\\
& = z^{-\frac{l+1}{2}}\left(\E^{-\I (\sqrt{z} - \frac{l \pi}{2})} + o\bigl(\E^{|\im(\sqrt{z})|}\bigr)\right).
\end{align*}
Hence, for $\im(z)\ge 0$ (implying $\im(\sqrt{z}) \ge0$) we obtain
\[
\frac{1}{\sqrt{z}} \beta(z) - \I \alpha(z) = z^{-\frac{l+1}{2}} \E^{\im(\sqrt{z})}
\left(\E^{-\I (\re(\sqrt{z}) - \frac{l \pi}{2})} + o(1)\right).
\]
Since $\alpha(z)$ and $\beta(z)$ have no common zeros this shows
\[
|\alpha(z)| + |\beta(z)| \ge B (1+|z|)^{-\frac{l+1}{2}} \E^{|\im(\sqrt{z})|} \ge B\E^{-|\sqrt{z}|}
\]
for $\im(z)\ge 0$. But since both sides of the inequality do not depend on the sign of $\im(z)$, the inequality
holds for all $z\in\C$. The claim now follows from this estimate as in the proof of Lemma~\ref{lem:cor}.
\end{proof}

In particular, the results of the previous sections apply to this example. For instance, Theorem~\ref{thmbm} now reads as follows.

\begin{theorem}
Let $q_0$ and $q_1$ be two potentials satisfying Hypothesis~\ref{hyp:q} with the same $l$ and
let $c\in(0,b)$. Choose $\phi_j(z,x)$ as in \eqref{asymphibes} and $\theta_j(.,x)\in R_{1/2}(\C)$ according to Lemma~\ref{lem:bes}.

Then, if $M_1(z)-M_0(z) = f(z) + O\big(\E^{-2(c-\eps) |\im(\sqrt{z})|}\big)$ for some $f(z)\in R_{1/2}(\C)$ and for every $\eps>0$ as $z\to\infty$ along any nonreal ray,
we have $q_0(x)=q_1(x)$ for $x\in(0,c)$.
\end{theorem}

\appendix

\section{The limit circle case}
\label{app:lc}

In this appendix we want to show that the singular Weyl $m$-function is a Herglotz--Nevanlinna function if one has the limit circle case at $a$.
We will use the same abbreviations as in Section~\ref{sec:swm}. The results should be compared with \cite{beev}, \cite{ful77}.

\begin{hypothesis}\label{hyp:qr}
Let $\phi_0(x)$ and $\theta_0(x)$ be two real-valued solutions of $\tau u = \lam_0 u$ for some fixed $\lam_0\in\R$
satisfying $W(\theta_0,\phi_0)=1$.

Assume that the limits
\be \label{limwr}
\lim_{x\to a} W_x(\phi_0,u(z)), \qquad \lim_{x\to a} W_x(\theta_0,u(z))
\ee
exist for every solution $u(z)$ of $\tau u = z u$.
\end{hypothesis}

\begin{remark}
Using the Pl\"ucker identity
\be\label{pluecker}
W_x(f_1,f_2) W_x(f_3,f_4) + W_x(f_1,f_3) W_x(f_4,f_2) + W_x(f_1,f_4) W_x(f_2,f_3) =0
\ee
it is not hard to see, that Hypothesis~\ref{hyp:qr} is independent of $\lam_0$.
\end{remark}

It is not hard to see that our hypothesis holds if $\tau$ is limit circle (l.c.) at $a$.

\begin{lemma}
Suppose $\tau$ is l.c.\ at $a$. Then Hypothesis~\ref{hyp:qr} holds. Moreover, the limits
\eqref{limwr} are holomorphic with respect to $z$ whenever $u(z,x)$ is.
\end{lemma}

\begin{proof}
Given two solutions $u(x)$, $v(x)$ of $\tau u = zu$, $\tau v = \hat{z} v$ it
is straightforward to check
\be\label{wronski}
(z-\hat{z}) \int_c^x u(y) v(y) dy = W_x(u, v) - W_c(u, v)
\ee
(clearly it is true for $x=c$, now differentiate with respect to $x$).
Hence the limit exists whenever $u$ and $v$ are square integrable near $a$ and
the first claim follows.

Moreover, this also shows that
\[
\lim_{x\to a} W_x(\phi_0,u(z)) = W_c(\phi_0,u(z)) - (\lambda_0-z) \int_a^c \phi_0(y) u(z,y) dy
\]
and to see the second claim it remains to show that the integral on the right-hand side is if $u(z,x)$
is holomorphic in a neighborhood of $z_0$. Then for all $z$ in a bounded neighborhood of $z_0$
we can estimate
\[
|u(z,x)| \le C_1 |c(z_0,x)| + C_1 |s(z_0,x)|, \quad
|u'(z,x)| \le C_3 |c'(z_0,x)| + C_4 |s'(z_0,x)|.
\]
This follows, for example, by inspecting the proof of \cite[Thm.~9.9]{tschroe}. Hence the limits \eqref{limwr}
are holomorphic in the same domain as $u(z,x)$.
\end{proof}

It will be shown below in Corollary~\ref{cor:lc} that Hypothesis~\ref{hyp:qr} is in fact equivalent to $\tau$ being l.c. at $x=a$.

So let $\tau$ satisfy Hypothesis~\ref{hyp:qr} and introduce
\begin{align}\nn
\phi(z,x) &= W_a(c(z),\phi_0) s(z,x) - W_a(s(z),\phi_0) c(z,x),\\ \label{defptqr}
\theta(z,x) &= W_a(c(z),\theta_0) s(z,x) - W_a(s(z),\theta_0) c(z,x).
\end{align}
Here the solutions $c(z,x)$ and $s(z,x)$ are defined as in Section~\ref{sec:swm}.
Clearly $\phi(z,x)^* = \phi(z^*,x)$ and $\theta(z,x)^* = \theta(z^*,x)$. Moreover,
$\phi(\lam_0,x)=\phi_0(x)$ and $\theta(\lam_0,x)=\theta_0(x)$.

\begin{lemma}\label{lem:a.4}
Suppose Hypothesis~\ref{hyp:qr}. Then $\theta(z,x)$ and $\phi(z,x)$ defined in \eqref{defptqr} satisfy
\be
W(\theta(z),\phi(z))= 1
\ee
and
\begin{align}\nn
& W_a(\theta(z),\phi(\hat{z}))= 1, \\
& W_a(\phi(\hat{z}),\phi(z))=W_a(\theta(\hat{z}),\theta(z))=0.
\end{align}
\end{lemma}

\begin{proof}
It suffices to show the last two identities.
\begin{align} \nn
W_a(\theta,\hat\phi) =& -W_a(\hat{c},\phi_0) \big( W_a(c,\theta_0) W_a(\hat{s},s) - W_a(s,\theta_0) W_a(\hat{s},c)\big)\\ \nn
&+ W_a(\hat{s},\phi_0) \big( W_a(c,\theta_0) W_a(\hat{c},s) + W_a(s,\theta_0) W_a(\hat{c},c)\big)\\ \nn
=& - W_a(\hat{c},\phi_0) W_a(\hat{s},\theta_0) + W_a(\hat{s},\phi_0) W_a(\hat{c},\theta_0)\\
=&  W(\hat{c},\hat{s}) W(\theta_0,\phi_0) = 1,
\end{align}
where we have twice used the Pl\"ucker identity \eqref{pluecker}
which remains valid in the limit $x \to a$. Similarly,
\begin{align} \nn
W_a(\hat\phi,\phi) =& W_a(\hat{c},\phi_0) \big( W_a(c,\phi_0) W_a(\hat{s},s) - W_a(s,\phi_0) W_a(\hat{s},c)\big)\\ \nn
&- W_a(\hat{s},\phi_0) \big( W_a(c,\phi_0) W_a(\hat{c},s) - W_a(s,\phi_0) W_a(\hat{c},c)\big)\\
=& W_a(\hat{c},\phi_0) W_a(\hat{s},\phi_0) - W_a(\hat{s},\phi_0) W_a(\hat{c},\phi_0) =0
\end{align}
and likewise for $W_a(\hat\theta,\theta)$.
\end{proof}

\begin{corollary}\label{cor:lc}
Suppose Hypothesis~\ref{hyp:qr} then $\theta(z,x)$ and $\phi(z,x)$ defined in \eqref{defptqr} satisfy
\begin{align}\nn
W_c(\phi(z),\phi(z)^*) &= 2\I\, \im(z) \int_a^c |\phi(z,x)|^2 dx, \\
W_c(\theta(z),\theta(z)^*) &= 2\I\, \im(z) \int_a^c |\theta(z,x)|^2 dx.
\end{align}

In particular $\tau$ is l.c.\ at $a$ and $\theta(z,x)$, $\phi(z,x)$ are entire with respect to $z$.
\end{corollary}

\begin{proof}
Choose $u=\phi(z)$ and $v= \phi(z)^*$ in \eqref{wronski}, let $x\to a$ and use Lemma \ref{lem:a.4} for the first formula.
For the second replace $\phi(z)$ by $\theta(z)$.
\end{proof}

\begin{lemma}
Suppose Hypothesis~\ref{hyp:qr} and let $H$ be some self-adjoint operator associated with $\tau$ and the
boundary condition induced by $\phi_0$ at $a$.

Then $\phi(z,x)$ defined in \eqref{defptqr} lies in the domain of $H$ near $a$, that is, $W_a(\phi(z),\phi_0)=0$.
Moreover,
\be \label{mmlc}
m_-(z)= \frac{W_a(\phi_0,c(z))}{W_a(\phi_0,s(z))}.
\ee
\end{lemma}

\begin{proof}
First of all it is easy to check $W_a(\phi(z),\phi_0)=0$. Moreover, using $0 = W_a(\phi_0,u_-(z))= W_a(\phi_0,c(z)) - m_-(z)
W_a(\phi_0,s(z))$ shows \eqref{mmlc}.
\end{proof}

Now we can introduce the singular Weyl $m$-function $M(z)$ as in Section~\ref{sec:swm} such that
\be\label{defswmlc}
\psi(z,x) = \theta(z,x) + M(z) \phi(z,x) \in L^2(c,b)
\ee
and $\psi(z,x)$ satisfies the boundary condition of $H$ at $b$ if $\tau$ is l.c.\ at $b$.

\begin{theorem}\label{thmMlc}
Suppose Hypothesis~\ref{hyp:qr} and let $H$ be some self-adjoint operator associated with $\tau$ and the
boundary condition induced by $\phi_0$ at $a$.

The singular Weyl $m$-function defined in \eqref{defswmlc} is a Herglotz--Nevanlinna function and satisfies
\be\label{imM}
\im( M(z) ) =  \im(z) \int_a^b |\psi(z,x)|^2 dx.
\ee
\end{theorem}

\begin{proof}
Now choose $c=a$, $u(x)=\psi(z,x)$ and $v(x)= \psi(z,x)^*$ in \eqref{wronski},
\[
-2\im(z) \int_a^x |\psi(z,y)|^2 dy = -\I W_x(\psi(z), \psi(z)^*) - 2 \im(M(z)),
\]
and observe that $W_x(\psi(z), \psi(z)^*)$ vanishes as $x\uparrow b$ since
both functions are in the domain of $H$ near $b$.
\end{proof}

Moreover, we can also strengthen Lemma~\ref{lemUub}.

\begin{lemma}
Suppose Hypothesis~\ref{hyp:qr} and let $H$ be some self-adjoint operator associated with $\tau$ and the
boundary condition induced by $\phi_0$ at $a$. Let $U$ be the associated spectral transformation as in Section~\ref{sec:st}.

Then
\be\label{Upsi}
(U \psi(z,.))(\lam) = \frac{1}{\lam-z}
\ee
for every $z\in\C\setminus\sig(H)$. Differentiating with respect to $z$ we even obtain
\be
(U \partial_z^k \psi(z,.))(\lam) = \frac{k!}{(\lam-z)^{k+1}}.
\ee
\end{lemma}

\begin{proof}
From Lemma~\ref{lemUub} we obtain
\[
(U \ti{\psi}(z,x,.))(\lam) = \frac{W_x(\theta(z),\phi(\lam))}{\lam-z},
\]
where
\[
\ti{\psi}(z,x,y) = W_x(\theta(z,.),G(z,.,y)) =\begin{cases} \psi(z,y), & y > x,\\
M(z) \phi(z,y), & y < x.
\end{cases}
\]
Now the claim follows by letting $x\to a$ using Lemma~\ref{lem:a.4}.
\end{proof}

The integral representation of $M(z)$ is given by:

\begin{corollary}
Under the same assumptions as in Theorem~\ref{thmMlc} we have
\be
M(z) = \re(M(\I)) + \int_\R \left(\frac{1}{\lam-z} - \frac{\lam}{1+\lam^2}\right) d\rho(\lam),
\ee
where $\rho$ is the spectral measure defined in Section~\ref{sec:st} which satisfies
$\int_\R d\rho(\lam)=\infty$ and $\int_\R \frac{d\rho(\lam)}{1+\lam^2}<\infty$.
\end{corollary}

\begin{proof}
Evaluating the right-hand side of \eqref{imM} using unitarity of $U$ and \eqref{Upsi} shows that both
sides have the same imaginary part. Since the real parts coincide at $z=\I$, both sides are equal.
Moreover, $\int_\R d\rho(\lam)=\infty$ and $\int_\R \frac{d\rho(\lam)}{1+\lam^2}<\infty$ follow from
$\psi(z) \not \in \dom(H)$ and $\psi(z) \in L^2(a,b)$, respectively.
\end{proof}

\begin{remark}
If $\tau$ is regular at $a$ we can choose $c=a$ and the approach presented here reduces
to the usual one (cf. \cite[Sect.~9.3]{tschroe}).
\end{remark}

\section{Another example}
\label{app:ex}

Our next example is connected with a class of explicit soliton-type solutions of  $\tau u= z u$,
and we again set $(a,b)=(0,\infty)$.
The constructions are similar to the ones in Theorem~1.2 and Proposition~1.3 of \cite{gks}, the main
difference being that here we will be interested in solutions with singularities at  $x=0$.

For simplicity we consider the scalar case though  formulas
  \eqref{BD2},  \eqref{BD3}, and the first equality in  \eqref{BD4} below are written in the standard form which holds also
for a general construction, where $A$ and $S$ are matrices; see \cite{gks}.

Let a vector $\begin{bmatrix} \upsilon_1 & \upsilon_2 \end{bmatrix} \in\C^2$ be given such that
\be \label{BD1}
0= \begin{bmatrix} \upsilon_1 & \upsilon_2 \end{bmatrix} J \begin{bmatrix} \upsilon_1 & \upsilon_2 \end{bmatrix}^*, \quad
J= \begin{bmatrix}
0 & 1 \\ -1 & 0
\end{bmatrix},
\ee
that is,
\[
\upsilon_1\upsilon_2^*=\upsilon_2\upsilon_1^*.
\]

Fix $A\in \C\setminus\{0\}$ and define the vector function  $\Lam(x)=\begin{bmatrix}\Lam_1(x) & \Lam_2(x) \end{bmatrix}$
as the solution of the linear equations with constant coefficients
\be \label{BD2}
\begin{cases}
\Lambda_1^{\prime  }(x)=  A  \Lambda_{2}  (x), \\
\Lambda_{2}^{\prime  }(x)=  -  \Lambda_{1}(x),
\end{cases}
\ee
which satisfies the following initial condition:
\be\label{eq_ic}
\Lam(0)=\begin{bmatrix} \upsilon_1 & \upsilon_2 \end{bmatrix}.
\ee
Define also the function
\be\label{BD2b}
S(x)=\int_0^x | \Lambda_{2}(t)|^{2}dt.
\ee
It follows from  \eqref{BD1}--\eqref{BD2b} that the identity
\be \label{BD3}
AS(x)-S(x)A^*=\Lam(x)J\Lam(x)^*, \qquad x \geq 0,
\ee
holds. Furthermore, it is straightforward to check that
\begin{align} \label{BD9}
\Lam_1(x)&= \upsilon_1\cos (\sqrt{A} x) +\upsilon_2\sqrt{A}\sin (\sqrt{A} x),\\
\Lam_2(x)&= \upsilon_2\cos (\sqrt{A} x) -\frac{\upsilon_1}{\sqrt{A}}\sin (\sqrt{A} x). \label{BD10}
\end{align}
For $x>0$  introduce the transfer matrix function $w_{A}$
in L. Sakhnovich form
\begin{equation}  \label{BD4}
w_{A}(z,x)=I_{2}+J \Lambda (x)^{*}S(x)^{-1}
(z - A)^{-1} \Lambda (x)
=\begin{bmatrix}
1+\frac{\Lam_2(x)^*\Lam_1(x)}{S(x)(z-A)} & \frac{|\Lam_2(x)|^2}{S(x)(z-A)}  \\  -\frac{|\Lam_1(x)|^2}{S(x)(z-A)}  & 1-\frac{\Lam_1(x)^*\Lam_2(x)}{S(x)(z-A)}
\end{bmatrix},
\end{equation}
where $I_2$ is the $2 \times 2$ identity matrix. Clearly, $w_A$ is well defined since $S(x)$ is positive for all $x>0$.
Then the potential $q$ corresponding to our triplet $\{\Lam_1(x),\Lam_2(x),S(x)\}$ is given by the formula
\be  \label{BD5}
q(x)= 2 \Big( \big(S(x)^{-1}|\Lambda_{2}(x)|^2\big)^2
+ S(x)^{-1}\big(\Lambda_{1}(x)^{*}\Lambda_{2}(x)+ \Lambda_{2}(x)^{*}
\Lambda_{1}(x) \big)\Big).
\ee
Introduce
\begin{equation}  \label{BD7}
T(z,x):=T_1(z)\begin{bmatrix}
\E^{\I\sqrt{z}x} & 0 \\ 0 & \E^{-\I\sqrt{z}x}
\end{bmatrix} T_1(z)^{-1},
\quad
T_{1}( z )=
\begin{bmatrix}
1 & 1  \\  \I \sqrt{z}  & -\I \sqrt{z}
\end{bmatrix}.
\end{equation}
Using \eqref{BD2}, \eqref{BD3}, \eqref{BD4}, and \eqref{BD5} one can show \cite{gks}
that the vector function
\begin{align}
y(z,x)=&
\begin{bmatrix} 1 & 0 \end{bmatrix}
w_{A}(z,x)T(z,x)\nn \\
=&\Big[1+\frac{\Lam_2(x)^*\Lam_1(x)}{S(x)(z-A)}\qquad \frac{|\Lam_2(x)|^2}{S(x)(z-A)}\Big] T(z,x),\label{BD6}\\
=&\begin{bmatrix} 1 & 0 \end{bmatrix}T(z,x)+\frac{\Lam_2(x)^*}{S(x)(z-A)}[\Lam_1(x)\ \ \Lam_2(x)]T(z,x)  \nn
\end{align}
satisfies $\tau u= z u$.
Note that
\be\label{BD7c}
T(z,x)=\begin{bmatrix}
\cos(\sqrt{z}x) & \sin(\sqrt{z}x)/\sqrt{z}  \\  -\sqrt{z}\sin(\sqrt{z}x)  & \cos(\sqrt{z}x)
\end{bmatrix}
\ee
is an entire matrix function of $z$.
Hence, $y(z,x)$ gives a system of linearly independent solutions, which are meromorphic in $z$
with only possible pole at $z=A$. Moreover, it is true that
\begin{equation}  \label{BD7'}
y^{\prime}(z,x )=
\begin{bmatrix} - S(x)^{-1}|\Lambda_{2}(x)|^2 & 1 \end{bmatrix}
w_{A}(z,x )T(z,x).
\end{equation}
It follows from \eqref{BD5} that $q$ may have a singularity at zero (it is shown below
that $q$ indeed has a singularity).
Moreover, according to  \eqref{BD2}, \eqref{eq_ic}, and \eqref{BD5}
$q$ does not change if we replace the initial condition $\Lam(0) = \begin{bmatrix}\upsilon_1 & \upsilon_2 \end{bmatrix}$
by $\Lam(0)= h \begin{bmatrix} \upsilon_1 & \upsilon_2 \end{bmatrix}$ ($h\not =0$).
Thus, without loss of generality  assume  that
\be\label{eq:aa}
\upsilon_1=\upsilon_1^*,\qquad \upsilon_2=1.
\ee
Now, consider the behavior  of the functions $\Lam_k$ and $S$ in a neighborhood
of  $x=0$.  By  \eqref{BD9} and \eqref{BD10} we get
\be\label{BD13}
\Lam_1(x) =\upsilon_1 +O(x), \quad \Lam_2(x) =1+O(x), \quad S(x)=x(1+O(x)).
\ee
Therefore, by \eqref{BD5}, we obtain
\be \label{BD18}
q(x)=\frac{2}{x^2}(1+O(x)).
\ee
So, we deal with a system with singularity.

Since $T(z,x)=I_2+O(x)$, it follows from \eqref{BD6} and  \eqref{BD13} that
\be\label{BD19b}
\ti \phi(z,x):= y(z,x)\begin{bmatrix}
1 \\ -\upsilon_1
\end{bmatrix} \in L^2(0,c).
\ee
Moreover, $\ti \phi$ satisfies $\tau u=z u$ since $y$ does.
Using the following equality
\[
T(z,x)=\cos(\sqrt{z}x)I_2+\frac{\sin(\sqrt{z}x)}{\sqrt{z}}
\begin{bmatrix}
0 & 1  \\  -z  & 0 \end{bmatrix},
\]
it is straightforward to check that
\begin{align}
\widetilde{\phi}(z,x)=\frac{1}{S(x)(z-A)}\Big(&\big(S(x)(z-A)+\Lam_2(x)^*\Lam_1(x)\big)\big(\cos(\sqrt{z}x)-\upsilon_1\frac{\sin(\sqrt{z}x)}{\sqrt{z}}\big)\Big. \nn\\
& \Big.- |\Lam_2(x)|^2\big(\upsilon_1\cos(\sqrt{z}x)+\sqrt{z}\sin(\sqrt{z}x)\big)\Big).
\label{eq:tiphi}
\end{align}
Clearly, $\widetilde{\phi}(z,x)$ is meromorphic with unique possible pole at $z=A$. However, it is
immediate from \eqref{BD9}, \eqref{BD10}, and \eqref{eq:aa} that
\begin{align} & \label{rr1}
\cos(\sqrt{z}x)-\upsilon_1\frac{\sin(\sqrt{z}x)}{\sqrt{z}}=\Lam_2(x)+O(z-A),
\quad z \to A;
\\ \label{rr2}&
 \upsilon_1\cos(\sqrt{z}x)+\sqrt{z}\sin(\sqrt{z}x)=\Lam_1(x)+O(z-A),
\quad z \to A.
\end{align}
Formulas \eqref{eq:tiphi}--\eqref{rr2} imply that
 $\widetilde{\phi}(z,x)$ is entire in $z$.

Let us show that
\be \label{BD21}
\phi(z,x):=(z-A^*)^{-1} \ti \phi(z,x)=(z-A^*)^{-1}y(z,x)
\begin{bmatrix}
1 \\ -\upsilon_1
\end{bmatrix}
\ee
is  real-valued on $\R$. To prove the claim it suffices to show that the function
\be\label{eq:F}
F(z):=S(x)(z-A)+\Lam_1(x)\Lam_2(x)^*
\ee
is real-valued on $\R$. The latter immediately follows from the identity \eqref{BD3}.

Next, consider the solution
\be\label{BD24}
\theta(z,x)=-(z-A)y(z,x)\begin{bmatrix} 0  \\  1 \end{bmatrix}.
\ee
Straightforward calculations give
\be \label{BD22}
\theta(z,x)=-\frac{\sin(\sqrt{z}x)}{S(x)\sqrt{z}}\Big(S(x)(z-A)+\Lam_2(x)^*\Lam_1(x)\Big) -\frac{\cos(\sqrt{z}x)}{S(x)}|\Lam_2(x)|^2.
\ee
Using the fact that the function $F(z)$ defined by \eqref{eq:F} is real for $z\in\R$,
we conclude that the solution $\theta(z,x)$ is entire in $z$ and real for $z\in\R$.

Now, let us  show that $W(\theta, \phi)=1$. Indeed, according to \eqref{BD6}--\eqref{BD7'}
we have
\begin{align}  \label{BD25}
W(\ta , \phi)&= \frac{A-z}{ z- A^*}
\det \left(w_{A}(z,x )T(z,x)
\begin{bmatrix}
0 &1  \\  1& -\upsilon_1
\end{bmatrix}\right)
\\ \nn
&=
 (z-A)(z-A^*)^{-1}
\det w_{A}(z,x ).
\end{align}
By \eqref{BD4}, using \eqref{BD3}, we
get
\be   \label{BD26}
\det w_{A}(z,x )=
(z- A)^{-1}(z- A^*),
\ee
and hence we derive $W(\ta , \phi)=1$.

The behavior of the discrete spectrum of the operators $H^X_{(0,c)}$, $X \in \{D,N\}$,
in this example follows from Theorem~2.5 in \cite{kst} and from Lemma~\ref{lem:hic}.
Solutions $\phi$ and $\theta$, which were constructed above, give an example of the entire solutions
that were discussed in Lemma~\ref{lem:theta}.  Finally, using these solutions we can construct
explicitly a singular Weyl $m$-function $M$. Namely, observe that
\[
T(z,x)\begin{bmatrix}
1  \\ \I \sqrt{z}
\end{bmatrix} =\begin{bmatrix}
\E^{\I\sqrt{z}x}  \\ \I\sqrt{z}\E^{\I\sqrt{z}x}
\end{bmatrix}\in L^2(c, \infty)
\]
for $\im(\sqrt{z}) >0$. Here  the branch cut in $\sqrt{z}$ is chosen along the negative real axis.
Since the transfer function $w_A(z,x)$ is bounded as $x\to +\infty$, we get
\begin{align}& \label{BD30}
\widetilde{\psi}(z,x)= y(x,z)\begin{bmatrix}
1  \\ \I \sqrt{z}
\end{bmatrix} = \begin{bmatrix}1 & 0 \end{bmatrix} w_A(z,x)\begin{bmatrix}
\E^{\I\sqrt{z}x}  \\ \I\sqrt{z}\E^{\I\sqrt{z}x}
\end{bmatrix}\in L^2(c, \infty).
\end{align}
Therefore, by \eqref{BD21}, \eqref{BD24}, and \eqref{BD30} we get
\begin{align}& \label{BD31}
 \psi(z,x)=\theta(z,x) +M(z)\phi(z,x) \in L^2(c, \infty), \quad M(z):=-\frac{(z-A)(z-A^*)}{\I \sqrt{z}+\upsilon_1},
\end{align}
that is, $M(z)$ above is a singular Weyl $m$-function of our system.
Thus, we proved the following lemma.

\begin{lemma}\label{La}
Let the parameters $A \in \C\setminus\{0\}$ and  $\upsilon_1 \in\R$ be  fixed and set $\upsilon_2=1$.
Then the relations \eqref{BD5}, \eqref{BD9}, \eqref{BD10}, and \eqref{BD2b}
explicitly determine a singular potential $q$ satisfying \eqref{BD18}.
The corresponding entire solutions $\phi(z,x)$ and $\theta(z,x)$, such that $\phi(z,x)$ is nonsingular
at $x=0$ and $W(\theta(z), \phi(z))=1$, are given by \eqref{BD21}, \eqref{eq:tiphi}, and \eqref{BD22},  respectively.

Moreover, the singular $m$-function corresponding to this problem is given by \eqref{BD31}.
\end{lemma}

\section{Generalized Nevanlinna functions}
\label{app:gnf}

In this appendix we collect some information on the classes $N_\kappa$
of generalized Nevanlinna functions \cite{krlan}. By $N_\kappa$,  $\kappa \in \N_0$, we denote the set of
all functions $M(z)$ which are meromorphic in $\C_+\cup \C_-$, satisfy the symmetry condition
\be\label{eq:B.01}
M(z)=M(z^*)^*\quad
\ee
for all $z$ from the domain $\mathcal{D}_M$ of holomorphy of $M(z)$,
and for which the Nevanlinna kernel
\be\label{eq:B.02}
\mathcal{N}_M(z,\zeta)=\frac{M(z)-M(\zeta)^*}{z-\zeta^*},\quad
z,\zeta\in\mathcal{D}_M,\ \ z\neq\zeta^*,
\ee
has $\kappa$ negative squares. That is,
for any choice of finitely many points $\{ z_j \}_{j=1}^n \subset \mathcal{D}_M$ the matrix
\be
\left\{ \mathcal{N}_M(z_j,z_k) \right\}_{1\le j,k\le n}
\ee
has  at most $\kappa$ negative eigenvalues and exactly $\kappa$ negative eigenvalues for some choice of $\{z_j\}_{j=1}^n$.
Note that $N_0$ coincides with the class of Herglotz--Nevanlinna functions.

Let $M\in N_\kappa$, $\kappa\ge 1$. A point $\lam_0\in\R$ is said to be a generalized pole of nonpositive
type of $M$ if either
\[
\limsup_{\eps\downarrow 0}\eps|M(\lam_0+\I\eps)|=\infty
\]
or the limit
\[
\lim_{\eps\downarrow 0}(-\I\eps)M(\lam_0+\I\eps)
\]
exists and is finite and negative. The point $\lam_0=\infty$ is said to be a generalized pole of nonpositive type of $M$
if either
\[
\limsup_{y\uparrow\infty}\frac{|M(\I y)|}{\I y}=\infty
\]
or
\[
\lim_{y\uparrow\infty}\frac{M(\I y)}{\I y}
\]
exists and is finite and negative. All limits can be replaced by nontangential limits.

We are interested in the special subclass $N_\kappa^\infty \subset N_\kappa$ of generalized Nevanlinna function with no
nonreal poles and the only generalized pole of nonpositive type at $\infty$. It follows from Theorem~3.1 (and its proof) and
Lemma~3.3 of \cite{krlan} that

\begin{theorem}
A function $M \in N_\kappa^\infty$ admits the representation
\be \label{minkappa}
M(z) = (1+z^2)^k \int_\R \left(\frac{1}{\lam-z} - \frac{\lam}{1+\lam^2}\right) \frac{d\rho(\lam)}{(1+\lam^2)^k}
+ \sum_{j=0}^l a_jz^j,
\ee
where $k\leq \kappa$, $l\le 2\kappa+1$,
\be\label{minkappa'}
a_j\in\R
\quad\text{and}\quad
\int_\R(1+\lam^2)^{-k -1}d\rho(\lam)<\infty.
\ee
The measure $\rho$ is given by the Stieltjes--Liv\v{s}i\'{c} inversion formula
\be
\frac{1}{2} \left( \rho\big((\lam_0,\lam_1)\big) + \rho\big([\lam_0,\lam_1]\big) \right)=
\lim_{\eps\downarrow 0} \frac{1}{\pi} \int_{\lam_0}^{\lam_1} \im\big(M(\lam+\I\eps)\big) d\lam.
\ee
The representation \eqref{minkappa} is called irreducible if $k$ is chosen minimal,
that is, either $k=0$ or $\int_\R(1+\lam^2)^{-k} d\rho(\lam)=\infty$.

Conversely, if \eqref{minkappa'} holds, then $M(z)$ defined via \eqref{minkappa} is in $N_\kappa^\infty$
for some $\kappa$. If $k$ is minimal, $\kappa$ is given by:
\be\label{eq:b.08}
\kappa=\begin{cases}
k, & l\le 2k,\\
\floor{\frac{l}{2}}, & l\ge 2k+1,\ l\ \text{even, or},\ l\ \text{odd and}\ a_l>0,\\
\floor{\frac{l}{2}}+1, & l\ge 2k+1,\ l\ \text{odd and},\ a_l< 0.\end{cases}
\ee
\end{theorem}

For additional equivalent conditions we refer to Definition~2.5 in \cite{DLSh}.

Given a generalized Nevanlinna function in $N_\kappa^\infty$, the corresponding $\kappa$
is given by the multiplicity of the generalized pole at $\infty$ which is determined by the facts
that the following limits exist and take values as indicated:
\[
\lim_{y\uparrow \infty} -\frac{M(\I y)}{(\I y) ^{2\kappa-1}} \in (0,\infty],
\qquad
\lim_{y\uparrow \infty} \frac{M(\I y)}{(\I y) ^{2\kappa+1}} \in [0,\infty).
\]
Again the limits can be replaced by nontangential ones.
This follows from Theorem~3.2 in \cite{lan}. To this end note that if $M(z) \in N_\kappa$, then $-M(z)^{-1}$,
$-M(1/z)$, and $1/M(1/z)$ also belong to $N_\kappa$. Moreover, generalized zeros of $M(z)$ are generalized poles
of $-M(z)^{-1}$ of the same multiplicity.

\begin{lemma}\label{lem:gnf:grow}
Let $M(z)$ be a generalized Nevanlinna function given by \eqref{minkappa}--\eqref{minkappa'} with $l<2k+1$.
Then, for every $0<\gam<2$, we have
\be
\int_\R \frac{d\rho(\lam)}{1+|\lam|^{2k+\gam}} < \infty \quad\Longleftrightarrow\quad
\int_1^\infty \frac{(-1)^k \im(M(\I y))}{y^{2k+\gam}} dy<\infty.
\ee
Concerning the case $\gam=0$ ,we have
\be
\int_\R \frac{d\rho(\lam)}{(1+\lam^2)^k} = \lim_{y\to\infty} \frac{(-1)^k \im(M(\I y))}{y^{2k-1}},
\ee
where the two sides are either both finite and equal or both infinite.
\end{lemma}

\begin{proof}
The first part follows directly from \cite[\S 3.5]{kk} (see also \cite[Lem.~9.20]{tschroe}). The second part follows by evaluating
the limit on the right-hand side using the integral representation plus monotone convergence (see e.g. \cite[\S 4]{kk}).
\end{proof}

\bigskip
\noindent
{\bf Acknowledgments.}
We thank Vladimir Derkach, Jonathan Eckhardt, Fritz Gesztesy, Daphne Gilbert, Annemarie Luger, and Mark Malamud for several helpful discussions.
A.K. acknowledges the hospitality and financial support of the Erwin Schr\"odinger Institute and financial support from the IRCSET PostDoctoral Fellowship Program.

\end{document}